\documentclass[10pt]{amsart}
\usepackage[utf8]{inputenc}
\usepackage[english]{babel}
\usepackage{tikz-cd}
\usepackage{graphicx}
\usepackage{amsthm}
\usepackage{amstext}
\usepackage{amsmath,amscd,amsfonts}
\usepackage{amssymb}
\usepackage{mathrsfs}
\usepackage{mathtools}
\usepackage[all]{xy}
\usepackage{stmaryrd}
\usepackage{color,comment}
\usepackage{amsmath, amsthm, amsfonts, enumerate}
\usepackage{amsmath}
\usepackage{graphicx}
\usepackage[colorlinks=true, allcolors=blue]{hyperref}

\newtheorem{theorem}{Theorem}[section]
\newtheorem{lemma}[theorem]{Lemma}
\newtheorem{corollary}[theorem]{Corollary}
\newtheorem{proposition}[theorem]{Proposition}
\newtheorem{definition}[theorem]{Definition}
\theoremstyle{remark}
\newtheorem{remark}[theorem]{\bf Remark}
\theoremstyle{definition}

\newtheorem{conjecture}[theorem]{Conjecture}
\numberwithin{equation}{section}

\DeclareMathOperator{\HS}{HS}
\newcommand{\kk}{\Bbbk}

\newcommand{\compBoundsixnine}{7}

\title[On the Hilbert series of ideals generated by general linear forms]{On the Hilbert series of ideals generated by powers of general linear forms}

\author{Mats Boij}

\address{Mats Boij, Department of Mathematics, KTH - Royal Institute of Technology,  SE-100 44 Stockholm, Sweden}
\email{boij@kth.se}

\author{Samuel Lundqvist}
\address{Samuel Lundqvist, Department of Mathematics, Stockholm University, SE-106 91 Stockholm, 
Sweden}
\email{samuel@math.su.se}

\begin{document}
\begin{abstract}
We determine the Hilbert series of ideals generated by $d$'th powers of $n+2$ general linear forms in $n$ variables, to give upper bounds on the degree of the Hilbert series of ideals generated by $d$'th powers of $n+k$ general linear forms for $k>2$. This allows us to show that the Iarrobino-Fr\"oberg Conjecture fails for all $n$ large enough. We also determine the degree of the Hilbert series for the ideal generated by $d$'th powers of $n+3$ general linear forms, for some values of $n$, and give counterexamples to a conjecture on the failure of the Weak Lefschetz Property for ideals generated by sufficiently large powers of general linear forms. Moreover, we determine the Hilbert series of the ideal generated by two generic quadratic forms in the exterior algebra on an even number of generators.    
\end{abstract}
\maketitle

\section{Introduction}

The \emph{Interpolation Problem} in algebraic geometry, which involves determining the dimension of the linear system of forms vanishing to specified orders at general points in the projective space, has lead to 
several conjectures that are still open in general. Nagata's Conjecture~\cite{NA59} from 1959 predicts a lower bound for the degree of such a form when there are at least $10$ points in the projective plane. Work by Segre, Harbourne, Gimigliano and Hirschowitz has successively made this more precise and their predictions about the dimension of such linear systems for points in the plane are known as the \emph{SHGH Conjecture}. Ciliberto~\cite{CI01} provides a survey over these developments. Using Macaulay's inverse systems, Emsalem and Iarrobino~\cite{EI95} showed how to relate the interpolation problem to the problem of determining the Hilbert function of ideals generated by powers of general linear forms. In 1985 Fr\"oberg~\cite{FR85} made a conjecture about the Hilbert function of ideals generated by general forms and in 1997 Iarrobino~\cite{I97} conjectured that ideals generated by powers of sufficiently many general forms should have the same Hilbert function. This has been known as the \emph{Iarrobino-Fr\"oberg Conjecture} and in 2007 Chandler \cite{C07} provided counterexamples to this conjecture $m=n+5$ general linear forms in $n=5,6,7$ variables. 
 
By a recent result by Laface, Postinghel, and
Santana S\'anchez \cite{LPS23} on linear systems, we are able to determine the Hilbert series of ideals generated by $d$'th powers of $n+2$ general linear forms. By carefully analyzing the first coefficient for which the series differs from the series for generic forms, we can prove the existence of non-trivial syzygies of $d$'th powers of $n+2$ general linear forms. We show that for $n$ large enough, these syzygies exists also in ideal generated by $d$'th powers of  $n+k$ general forms, for $k>2$, from which it follows that the Hilbert series of an ideal generated by $d \geq 2$'th powers of linear forms of degree $k\geq 2$ differs from the series conjectured for $k$ generic forms of degree for all $n$ large enough. The proof for this statement uses the central limit theorem from probability theory.

We also use a connection between squares of $n+2$ linear forms and two generic quadratic forms in the exterior algebra \cite{CLN19} to resolve a combinatorial conjecture on the Hilbert series of the exterior algebra on an even number of generators modulo two generic quadratic forms.

A well known technique used for studying  linear systems are Cremona transformations. These techniques have been used also for studying ideals generated by powers of general linear forms with respect to the Lefschetz properties by means of the Emsalem-Iarrobino correspondence. 
We state and prove an algebraic version of the Cremona transformation and then use it to draw conclusions on the Hilbert series of $d$'th powers of $n+3$ general forms for $n=3,4,5,6$, and for both $n+4$ and $n+5$ general forms for $n=3,4$. In connection to these results we are able to give counterexamples to a conjecture by Harbourne, Schenck, and Seceleanu \cite{HSS11} on the weak Lefschetz property.

\section{Preliminaries} \label{sec:prel}

Throughout the rest of the paper, we let $\kk$ be a field of characteristic zero and we let $S=\kk[x_1,\ldots,x_n]$ be the standard graded polynomial ring, and 
\[R_{n,m,d} = S/(\ell_1^d,\ldots,\ell_m^d),\] where each $\ell_i$ is a general linear form.

We write $\HS(A,t)$ for the Hilbert series of a standard graded algebra $A$, and we
let $D_{n,m,d} = \deg (R_{n,m,d}).$

We use the convention that $\binom{n}{-i} = \binom{-n}{i} = \binom{-n}{-i}=0$ for positive integers $i$ and $n$.

\subsection{The dimension of the linear system}

Let $p_1,\ldots,p_s$ be points in $\mathbb{P}^n$ and let $m_1,\ldots, m_s$ be positive integers. 
The polynomial interpolation problem is the problem of determining the dimension of the linear system of forms of degree $d$ that vanish to the order at least $m_i$ at $p_i$, for $i=1,\ldots,s.$ For general points $p_1,\ldots,p_s$, this is denoted $\mathfrak{L}_{n,d}(m_1,\ldots,m_{s})$, and it is well known that it equals the dimension of the $d$'th graded component of the fat point ideal 
\[ \mathfrak{p}_1^{m_1} \cap 
\cdots \cap \mathfrak{p}_{s}^{m_{s}} \subset \kk[x_0,\ldots,x_{n}],\]
where $\mathfrak{p_1},\ldots, \mathfrak{p_{s}}$ are the ideals vanishing on the points $p_1,\ldots,p_{s}$, respectively.
 
The problem is open in general. The 
Alexander-Hirschowitz \cite{ah} theorem settles the case of general points and $m_1=\ldots=m_s =2$. Laface, Postinghel, and Santana S\'anchez \cite{LPS23} solve the problem when the points lie on a rational normal curve. The latter result in the special case of $n+3$ points is the following.

\begin{theorem} \cite[Theorem 2.1 restricted to the case of $n+3$ points on a rational normal curve]{LPS23} \label{thm:linearsystem}
Suppose that $m_j \geq \sum_{j=1}^{n+3} m_j - ni$ for $j \in \{1,\ldots,n+3\}.$
Then 
\begin{align*}
\dim(\mathfrak{L}_{n,i}(m_1,\ldots,m_{n+3})) =\\  
\sum_{t=0}^{\lfloor n/2 \rfloor } 
\sum_{\substack{
I \subset \{1,\ldots,n+3\} \\
0 \leq |I| \leq n-2t
}}
 (-1)^{|I|} \binom{ n+ \sum_{j \in I} m_j + t \sum_{j=1}^{n+3} m_j - t ni - ti -|I|i + i - |I|-2t}{n}.
\end{align*}
\end{theorem}

We note that the condition on the multiplicities in Theorem \ref{thm:linearsystem} can be resolved, but as we will not need it for our arguments, we refer the reader to \cite{LPS23} for details.

Recall that there is a rational normal curve passing through $n+3$ general points in $\mathbb{P}^n$. Thus
\[ \dim(\mathfrak{L}_{n,i}(m_1,\ldots,m_{n+3})) = \dim_{\kk} I_i, \] 
where \[ I =\mathfrak{p}_1^{m_1} \cap  
\cdots \cap \mathfrak{p}_{n+3}^{m_{n+3}} \subset 
\kk[x_0,\ldots,x_{n}], \]
and where the corresponding points $p_1,\ldots,p_{n+3}$ are general.

\subsection{The inverse system}

Consider the dual polynomial ring $S'=\kk[X_1,\dots,X_n]$, where $x_i \in S$ acts like $\partial /\partial X_i$, for $i=1,\dots,n$. The inverse system of an ideal $I \in S$, denoted by $I^{-1}$, is the submodule annihilated by $I$ under this action. We use the notation $f\circ F$ for the action of the form $f\in S$ on the form $F\in S'$. By duality,
\begin{equation} \label{eq:inverse} \dim_\kk [I^{-1}]_d =
  \dim_\kk[S/I]_d.
\end{equation}
Will abuse notation and in the rest of the paper we will denote the variables in the dual ring by $x_1,\ldots,x_n$, instead of $X_1,\ldots,X_n.$

Let $p_1,\ldots,p_s$ be points in $\mathbb{P}^{n-1}$ and let $m_1,\ldots, m_s$ be positive integers. Let 
\[ I =\mathfrak{p}_1^{m_1} \cap 
\cdots \cap \mathfrak{p}_{s}^{m_{s}} \subset 
\kk[x_1,\ldots,x_{n}]. \]
Emsalem and Iarrobino \cite{EI95} showed that
\[(I^{-1})_i = 
\begin{cases}
    S'_i & \text{ if } 0 \leq i < \max(m_1,\ldots,m_s) \\
    ( \ell_1^{i-m_1+1}, \ldots, \ell_s^{i-m_s+1} )_i & \text{ if } i \geq \max(m_1,\ldots,m_s).
\end{cases}, \]
where the coefficients of the linear form $\ell_i \in S'$ are given by the coordinates of the point $p_i$. In particular, letting $a_i$ denote the dimension of the $i$'th graded component of $R_{n,n+2,d}$, we have

\begin{equation} \label{eq:emsalemiarrobino}
a_i  = 
\begin{cases}
\binom{n-1+i}{n-1} & \text{ if } 0 \leq i < d  \\
      \dim \mathfrak{L}_{n-1,i}((i-d+1)^{n+2}) & \text{ if } i \geq d,
\end{cases}
\end{equation}
where $\mathfrak{L}_{n-1,i}(b^c)$ is short for 
$\mathfrak{L}_{n-1,i}(\underbrace{b,\ldots, b}_{c \text{ times}}).$

\subsection{The Fröberg conjecture}

The Fr\"oberg conjecture states that the Hilbert series of $S/(f_1,\ldots, f_m),$ where the $f_i$ are generic forms of degrees $d_1,\ldots,d_m$, equals
\begin{equation} \label{eq:froberg}
\left[\frac{\prod (1-t^{d_i})}{(1-t)^n} \right]
\end{equation}
where the brackets means truncate at the first non-positive term. It is straightforward to check that up to this first non-positive term, the coefficient in front of $t^j$ in \ref{eq:froberg} equals
\begin{equation}
\sum_{k=0}^m (-1)^k \sum_{1\le
    i_1<i_2<\dots<i_k} \binom{j-d_{i_1}-\cdots-d_{i_k}+n-1}{n-1},\end{equation}
    which in the equigenerated case $d=d_1=\ldots=d_m$
    simplifies to
\begin{equation}
\sum_{k \geq 0} (-1)^k \binom{m}{k} \binom{j-kd+n-1}{n-1}.
\end{equation}

The conjecture is settled in the cases $n \leq 3$ \cite{FR85,AN86}, and in the case $m=n+1$ \cite{ST80}. In the equigenerated case $d=d_1=\ldots=d_m$ it is also known that the coefficients in the series (\ref{eq:froberg}) are correct up to degree $d+1$ for any $d$ \cite{HL87}, and up to degree $d+2$ for any $d>2$, and in general, correct up to degree $2d+1$ for $d$ large enough \cite{BDL26}. 

\subsection{The Iarrobino-Fr\"oberg conjecture, the SHGH conjecture, and Nagata's conjecture}

Iarrobino \cite{I97} made the following conjecture, known in the literature  as the Iarrobino-Fr\"oberg Conjecture.

\begin{conjecture}
The Hilbert series of $R_{n,m,d}$ agree with the series predicted by Fr\"oberg for $m$ general forms of degree $d$, except for the cases $m= n + 2, m= n + 3$, and $(n,m) \in \{(3,7),(3,8), (4,9), (5,14)\}$.
\end{conjecture}
As mentioned in the introduction, Chandler \cite{C07}, gave counterexamples to the Iarrobino-Fr\"oberg conjecture for $(n,m) \in \{
(5,10), (6,11), (7,12) \}$ for $d$ large enough. One of the aims of the present paper is to give a deeper understanding of the failure of this conjecture. 

The SHGH conjecture is a conjecture on the Hilbert series of general fat point ideals
\[ \mathfrak{p}_1^{d_1} \cap 
\cdots \cap \mathfrak{p}_{m}^{d_{m}} \subset \kk[x_0,x_1,x_{2}]\] for  $m \geq 10$. For so called quasi uniform degrees $d_1 \geq  \cdots = d_9 \geq d_{10} \geq \cdots \geq d_{m},$ Harbourne, Holay, and Fitchett \cite{HHF03} showed that the conjecture takes the form 
\[ \dim_{\kk} (\mathfrak{p}_1^{d_1} \cap 
\cdots \cap \mathfrak{p}_{m}^{d_{m}})_i = \max \left(0,\binom{i+2}{2} - \sum_i \binom{d_i+1}{2} \right). \]

By the Emsalem-Iarrobino correspondence, this means that  \[ \dim_{\kk} (\kk[x_1,x_2,x_3]/(\ell_1^{d_1}, \ldots, \ell_m^{d_m}))_i = \max \left(0,\binom{i+2}{2} - \sum_i \binom{i-d_i}{2} \right). \] 

In Proposition \ref{prop:39}, we determine the degree of $R_{3,9,d}$, and in particular, we show that its degree is below the Koszul degree. This means that the Fr\"oberg's conjecture for the dimension in degree $i$ of the quotient by $m \geq 9$ forms of degrees $d_1=\cdots=d_9 \leq d_{10} \leq \ldots \leq d_m$ becomes
\[
\binom{i+2}{2} - \sum_i \binom{i-d_i}{2} 
\]
as long as it is positive, and then equal to zero. Since the expression is quadratic, it follows that  Fr\"oberg's conjecture becomes
\[
\max \left(0,\binom{i+2}{2} - \sum_i \binom{i-d_i}{2} \right) 
\]
in this case. We conclude that SHGH conjecture for quasi uniform degrees is equivalent to the statement that the corresponding algebra of general linear forms has the series predicted by Fr\"oberg for generic forms of the same degrees. In particular, in the uniform case $d_1= \cdots = d_m = d$, it coincides with the Iarrobino conjecture. Ciliberto and Miranda \cite{CM06} proved that for any $d$, $m$ and $k$, the linear system of plane curves f degree $d$ having $k^2$ general multiple points of multiplicity $m$ has the expected dimension, which proves the SHGH conjecture in the equigenerated case for squares. In our language, this means that $R_{3,k^2,d}$ has the Hilbert series predicted by Fröberg for all $k\ge 2$. 

In 1959, Nagata \cite{NA59} conjectured that for a curve of degree $r$ in 
$\mathbb{P}^2$ to pass through $m \geq 10$ general points $p_1,\ldots,p_m$ of multiplicities $d_1,\ldots,d_m$, the inequality  \[r \sqrt{m} \geq \sum d_i \] 
should hold. Nagata showed the inequlity in case when $m$ is a square.

For us this means that the degree of $R_{3,m,d}$ should be at most 
\[ \left \lceil \frac{(d-1) \sqrt{m}}{\sqrt{m}-1 } \right \rceil - 1. \] 
Ciliberto \cite{CI01} showed that the SHGH conjecture implies Nagata's conjecture.

\subsection{The Lefschetz properties}
An artinian graded algebra $A$ has the weak Lefschetz property (WLP) if $\HS(A/(\ell),t) = [(1-t)\HS(A,t)]$ for a general linear form $\ell$, and it has the strong Lefschetz property (SLP) if $\HS(A/(\ell^i),t) = [(1-t)^i\HS(A,t)]$ for all $i$. It is well known that monomial complete intersections have the SLP, and this result goes back to Stanley \cite{ST80}. Recall that monomial complete intersections having the SLP implies the Fr\"oberg conjecture for the case $m = n + 1$. 

It is also known that $R_{n,n+1,d}$ does not have the WLP \cite{BL23}, except for the cases  $n \leq 3$, $d=1$, or $(n,d) \in \{(4,2),(5,2),(5,3),(7,2)\}.$ The method of proof used in \cite{BL23} to derive this result was to show that the degree of the Hilbert series of $R_{n+1,n+2,d}/(\ell) = R_{n,n+2,d}$ is greater then the degree of the Hilbert series predicted by Fr\"oberg for $n+2$ generic forms of degree $d$. 

The present paper is a continuation of the paper \cite{BL23}, and although not as focused on the WLP as the previous one,  in Theorem \ref{thm:diffdeg} we will strengthen the main result of \cite{BL23} by also providing the minimal degree for which the multiplication map fails to have maximal rank. 

Harbourne, Schenck and Seceleanu \cite{HSS11} have conjectured that $R_{n,m,d}$ fails the WLP when $m \geq n+1 \geq 5$ for $d >>0$. We show that the algebras $R_{4,k^2,d}$ and $R_{5,8,d}$ have the WLP for all $k \geq 2$ and all $d$, thus providing counterexamples to the conjecture.

\subsection{The Cox-Nagata ring} 
For linear forms $\ell_1,\ldots, \ell_m$ and positive integers $c_1,\ldots,c_m,d$, consider 
\[V_{c_1,\ldots,c_m,d}  = \{ F \in S_d : \ell_i^{c_i+1} \circ F= 0\} \]  
and define 
\[C = \bigoplus_{c_1,\ldots,c_m,d} V_{c_1,\ldots,c_m,d}.\]
Now $C$ is a graded ring since 
\[ \ell_i^{c_i + c'_i + 1} \circ FG = 0 \]  
if $F \in V_{c_1,\ldots,c_m,d}$ and $G \in V_{c'_1,\ldots,c'_m,d'}$,
showing that \[V_{c_1,\ldots,c_m,d} V_{c'_1,\ldots,c'_m,d'} \subseteq 
V_{c_1+c'_1,\ldots,c_m+c'_m,d+d'}.\]

This ring structure is what governs the failure of the WLP for 
$R_{n,n+1,d}$ when $d\geq 3$. 

Sturmfels and Xu \cite{SX10} calls $C$ the \emph{Cox-Nagata ring}. Nagata showed \cite{NA59} that for $n=3, m=16$, the ring $C$ is not finitely generated, which gave a negative answer to Hilbert's 14'th problem. It is now known that $C$ is finitely generated if and only if 
\[ \frac{1}{2} + \frac{1}{n} + \frac{1}{m-n} > 1, \]
see \cite{SX10} who attributes the result to \cite{CT06,Mu05}. For a description of the generators in some of the finitely generated cases, see \cite{BM26} and the references therein.

\section{The Hilbert series of $n+2$ $d$'th powers of general linear forms}

We begin by applying Theorem \ref{thm:linearsystem} to determine the Hilbert series of $R_{n,n+2,d}$. 

\begin{theorem} \label{thm:n+2equi}
The Hilbert series of $R_{n,n+2,d}$ is given by the polynomial
\[ 
\sum_{i=0}^{D_{n,n+2,d}} a_i t^i,
\]
where \begin{equation} \label{eq:ai}
    a_i = \sum_{r=0}^{\lfloor (n-1)/2 \rfloor } \sum_{k=0}^{n-1-2r} (-1)^{k} 
\binom{n+2}{k}  \binom{- n r d + n r + 2 i r - 2 r d - k d + n + i - 1}{n-1},
\end{equation}    
and where 
\[D_{n,n+2,d} = 
\begin{cases}
(n+1)(d-1)/2 & \text{if $n$ is odd}\\
\lfloor \frac{n(n+2)(d-1)}{2(n+1)} \rfloor & \text{if $n$ is even}\\
\end{cases} \]
\end{theorem}

\begin{proof}  
Assuming that $i \geq d$ we get by the Emsalem-Iarrobino correspondence (\ref{eq:emsalemiarrobino}) that the dimension in degree $i$ equals
\[\dim(\mathfrak{L}_{n-1,i}((i-d+1)^{n+2}). \]
Theorem \ref{thm:linearsystem} applied to the equigenerated case gives
\begin{align*}
\dim(\mathfrak{L}_{n-1,i}((i-d+1)^{n+2}))=&\\
\sum_{r=0}^{\lfloor (n-1)/2 \rfloor } \sum_{k=0}^{n-1-2r} (-1)^{k} 
\binom{n+2}{k}  \binom{- n r d + n r + 2 i r - 2 r d - k d + n + i - 1}{n-1}&,
\end{align*}
under the assumption  \[i-d+1 \geq (i-d+1)(n+2)-(n-1)i \iff i \leq (d-1)(n+1)/2.\]

By \cite{BL23} the degree of the series of $R_{n,n+2,d}$ equals $D_{n,n+2,d}$, and since $D_{n,n+2,d} \leq (d-1)(n+1)/2$, it follows that 
 (\ref{eq:ai}) holds for 
$i \in \{d,d+1,\ldots,D_{n,n+2,d}\}$.

Suppose now that $i < d$, in which case the dimension in degree $i$ is equal to that of the polynomial ring in degree $i$. Thus we need to check that the right hand side of \ref{eq:ai} equals $\binom{ n + i - 1}{n-1}.$ The contribution for $t=0$ is
\[\sum_{k=0}^{n-1} (-1)^{k} 
\binom{n+2}{k}  \binom{- k d + n + i - 1}{n-1} = 
\binom{ n + i - 1}{n-1}, \]
since $- k d + n + i - 1 < n-1$ when $k \geq 1.$ From the observation that $-n r d + nr + 2ir- 2rd - kd + n + i - 1$ is decreasing with $t$, it suffices to check that there is no contribution for $t=1$ and $k=0$, which is straightforward. This concludes the proof.

\end{proof}
We now introduce the notation $D_{n,m,d}$ for the degree of the Hilbert series of $R_{n,m,d}$, and we let 
\[g(n,d) = \left \lfloor \frac{(n+2)(d-1)+1}{3} 
\right \rfloor. \]
We will show that the series for $R_{n,n+2,d}$ agrees with that predicted by Fr\"oberg until degree $g(n,d)$, and that there are non-trivial syzygies in degree  $g(n,d)+1$, provided that there is room for them, that is, that $D_{n,n+2,d} \geq  g(n,d)+1$. We first need two lemmas. 

\begin{lemma} \label{lemma:sr}
Let $n \geq 3, d \geq 2.$ Then  
\[g(n,d) \leq D_{n,n+2,d}, \]
with equality if and only if 
\[(n,d) \in \{(3,2), (4,2), (6,2), (4,3)\}.\]
\end{lemma} 

\begin{proof}
Let $n$ be odd. It is straightforward to check that \[ \frac{(n+2)(d-1)+1}{3} + 1\leq (n+2)(d-1)/2 \iff 8 \leq (n-1)(d-1). \] 
When $d=2$ this holds for $n \geq 8$, when $d=3,4$ it holds for $n\geq 5$,  when $d\geq5$ it holds for $n\geq 3$.
We are left with the case $(n,d) = (3,2)$ where we verify equality, and with the cases $(n,d) \in \{ (3,4),(5,2),(7,2)\}$ where we verify inequality.

Let $n$ be even. This time we have
 \[ \frac{(n+2)(d-1)+1}{3}+ 1 \leq \frac{n(n+2)(d-1)}{2(n+1)} -1 \iff 14 \leq (n-2)(n+2)(d-1)/(n+1). \] For $d=2$ the inequality holds for $n \geq 16$, for $d=3$ the inequality holds for $n \geq 10$, for $d=4$ the inequality holds for $n \geq 8$, for $d=5,6$, it holds for $n \geq6$ and for $d\geq 7$ the inequality holds for $n \geq 4$.
We are left with the cases $(n,d) \in \{(4,2),(6,2),(4,3)\}$ where we verify equality, and with the cases $(n,d) \in \{(8,2),(10,2),(12,2),(14,2),(6,3), (8,3),(4,4),(6,4),(4,5),(6,5)\}$ where we verify inequality.
\end{proof}

\begin{lemma} \label{lemma:differ}
Let $n \geq 3, d \geq 2.$ \begin{enumerate}
\item 
For $i \leq g(n,d)$ we have
\[a_i > 0 \text{ and } a_i =
\sum_{k = 0}^{n-1} (-1)^{k} 
\binom{n+2}{k}  \binom{ n - 1 + i - k d}{n-1}.
\] 

\item For $ i= g(n,d) + 1$, and $(n,d) \notin \{(3,2), (4,2), (6,2),(4,3) \}$, we have

\[a_i  > 0 \text{ and } a_i =
\sum_{k = 0}^{n-1} (-1)^{k} 
\binom{n+2}{k}  \binom{ n - 1 + i - k d}{n-1} + \epsilon(n,d),
\]
where 
\[\epsilon(n,d) = \begin{cases}
n & \text{ if } (n+2)(d-1) \equiv 0 \pmod{3}\\
1 & \text{ if }  (n+2)(d-1) \equiv 1 \pmod{3} \\
\binom{n+1}{2} - n-2 & \text{ if } d = 2,  (n+2)(d-1) \equiv 2 \pmod{3}\\
\binom{n+1}{2} & \text{ if } d\geq 3,  (n+2)(d-1) \equiv 2 \pmod{3}.

\end{cases}
\]
\end{enumerate}
\end{lemma}

\begin{proof}
The statement about the positivity of the coefficients follows by
Lemma \ref{lemma:sr} and Theorem \ref{thm:n+2equi}, so we have that
\[a_i =\sum_{r=0}^{\lfloor (n-1)/2 \rfloor } \sum_{k=0}^{n-1-2r} (-1)^{k} 
\binom{n+2}{k}  \binom{- n r d + n r + 2 i r - 2 r d - k d + n + i - 1}{n-1}>0\]
for 
\[
i \leq \begin{cases} 
g(n,d) + 1 & \text{ for }(n,d) \notin \{(3,2),(4,2),(6,2),(4,3)\}\\  
g(n,d) & \text{ for }(n,d) \in \{(3,2),(4,2),(6,2),(4,3)\}.\\
\end{cases}
\]

We now examine the expression \[-nrd + nr + 2ir -2rd - kd + n + i - 1\] for  $r\geq 1$.
When $r=1$ we get
\[- n d + n + 2 i- 2 d - k d + i \geq 0 \iff
i \geq ((n+2)(d-1)+2 +kd)/3.\] The right hand side is increasing with $k$. For $k=0$ we get after simplification that \[i \geq \lfloor (n+2)(d-1)+1) /3 \rfloor + 1 = g(n,d)+1.\]
When $r=2$ we get that 
\[- 2 n  d + 2 n  + 4 i - 4 d - k d + i \geq 0 \iff
i \geq 2(n+2)(d-1)+2 +kd)/5.\]
Setting $k = 0$ this implies \[i \geq 2(n+2)(d-1)+2)/5.\]
We claim that \[2(n+2)(d-1)+2)/5 - (g(n,d) + 1) \geq 1 \]
which follows from the weaker statement that
\[2(n+2)(d-1)+2)/5-
(n+2)(d-1)+2)/3 \geq 1
\iff
(n+2)(d-1)+2 \geq 15\] which for $d=2$ holds when $n\geq 11$,
for $d=3$ holds when $n \geq 5$, and for $d \geq 4$ holds when $n \geq 3.$ For the remaining cases $d=2, n \in \{5,7,8,9,10\}$ and $(n,d) =(3,3)$ one can check that 
\[g(n,d) + 1< 2(n+2)(d-1)+2 +kd)/5. \]
This means that there is no contribution from $r=2$ in the expression for $a_i$ when $i = g(n,d) + 1$.  Since $- n r d + n r + 2 i r - 2 r d - k d + i$ is decreasing with $r$, there cannot be a contribution for $r \geq 3$ either. We get that 
\[a_i = \sum_{k = 0}^{n-1} (-1)^{k} 
\binom{n+2}{k}  \binom{ n - 1 + i - k d}{n-1}\] when 
$i \leq g(n,d)$ and that
\begin{multline} \label{eq:t1}
a_i = \sum_{k = 0}^{n-1} (-1)^{k} 
\binom{n+2}{k}  \binom{ n - 1 + i - k d}{n-1} + \\
\sum_{k = 0}^{n-1} (-1)^{k} 
\binom{n+2}{k}  \binom{- n d + n + 3 i- 2 d - k d +n-1}{n-1}
\end{multline}
when
$i =g(n,d)+ 1.$

Suppose that $(n+2)(d-1)+2$ is divisible by $3$. Then 
$i=g(n,d) + 1 = ((n+2)(d-1)+2)/3,$ so $- n d + n + 3 i- 2 d - k d  = -kd.$ Thus only $k = 0$ gives a non-zero contribution, so the right summand of Equation \ref{eq:t1} equals $\binom{n-1}{n-1}.$

If $(n+2)(d-1)+2 \equiv 1 \pmod{3}$, then 
$i = g(n,d)+1= ((n+2)(d-1)+4)/3$, so
$- n d + n + 3 i- 2 d - k d = 2-kd.$ If $d \geq 3$ then only $k=0$ gives a non-zero contribution equal to $\binom{n+1}{n-1}$. If $d=2$ also $k=1$ will contribute with the term $-\binom{n+2}{1}\binom{n-1}{n-1}.$ 

If $(n+2)(d-1)+2 \equiv 2 \pmod{3}$, then
$i =  g(n,d)+1 = ((n+2)(d-1)+3)/3$, so
$- n d + n + 3 i- 2 d - k d = 1-kd$, giving the contribution $\binom{n}{n-1}$ when $k=0$.
\end{proof}

The result we are about to derive is tightly connected to the classification of the WLP for $R_{n,n+1,d}$, and for completeness, we add the corresponding statement for the WLP. In fact, we improve  \cite[Theorem 1.1]{BL23} in the case of failure by giving the minimal failing degree.

\begin{theorem} \label{thm:diffdeg}

For $n \leq 2$ or $d=1$ or $(n,d) \in 
\{(3,2), (4,2), (4,3), (6,2) \}$, the Hilbert series for $R_{n,n+2,d}$ is equal to the series predicted by Fr\"oberg for $n+2$ generic forms.

For $n \geq 3,d \geq 2$ and $(n,d) \notin 
\{(3,2), (4,2), (4,3), (6,2) \}$, the coefficients in the Hilbert series for $R_{n,n+2,d}$ equal the coefficients in the series predicted by Fr\"oberg for $n+2$ generic forms of degrees equal to $d$ until degree 
\[g(n,d)=\left \lfloor \frac{(n+2)(d-1)+1}{3} \right \rfloor,\] while in the next degree, the coefficients of the two series differ from each other.

In terms of the WLP, this means that 
$R_{n,n+1,d}$ fails the WLP for the first time in degree 
$g(n,d)$
except when $n \leq 3$ or $d=1$ or $(n,d) \in \{(4,2), (5,2), (5,3), (7,2)\}$, and in these cases, the WLP holds.   
\end{theorem}

\begin{proof}
Let $\ell$ be a general linear form. Since 
$R_{n,n+1,d}/(\ell) \cong R_{n-1,n+1,d}$, it is clear that $R_{n,n+1,d}/(\ell)$ fails the WLP if and only if the series differ from the one predicted by Fr\"oberg for $R_{n-1,n+1,d}$. Moreover, if the series differ and $i$ is the maximal degree for which $R_{n-1,n+1,d}$ is equal to the series predicted by Fr\"oberg, then $i$ is the minimal degree for which the WLP fails.

By the classification in \cite[Theorem 1.1]{BL23}, the WLP holds for $R_{n,n+1,d}$ if and only if  $n \leq 3$ or $d=1$ or $(n,d) \in \{(4,2), (5,2), (5,3), (7,2)\},$ so the series for $R_{n,n+2,d}$ equals the one predicted by Fr\"oberg if and only if $n \leq 2$ or $d=1$ or $(n,d) \in \{(3,2), (4,2), (4,3), (6,2) \}.$

It remains the show that for the remaining cases, the Hilbert series for $R_{n,n+2,d}$ is equal to the series predicted by Fr\"oberg  until degree
\[g(n,d)=\left \lfloor \frac{(n+2)(d-1)+1}{3} \right \rfloor,\] 
while in degree $g(n,d)+1$, the two series differ from each other.

Suppose that $n \geq 3$, and that $i \leq g(n,d)$. The series predicted by Fr\"oberg is 
\[ \left[ \frac{(1-t^d)^{n+2}}{(1-t)^n} \right] =: \sum_i b_it^i.  \]
It is straightforward to check that
\[b_i = \begin{cases}
    1 & \text{if } i=0 \\
    \max(0,\sum_{k \geq 0} (-1)^{k} 
\binom{n+2}{k}  \binom{ n - 1 + i - k d}{n-1}) & \text{if $i \geq 1$ and $b_{i-1}$ is positive} \\
0 & \text{otherwise.} 
\end{cases}\]

The degree of $\sum b_i t^i$ is bounded above by $D_{n,n+2,d}$, and $D_{n,n+2,d} \leq (n+1)(d-1)/2$. Thus $i \leq (n+1)(d-1)/2$, so $i-kd \geq 0$ gives $k \leq (n+1)/2$, which implies that $k \leq n-1$ since $n \geq 3.$ Thus 

\[
\sum_{k \geq 0} (-1)^{k} 
\binom{n+2}{k}  \binom{ n - 1 + i - k d}{n-1}
=\sum_{k \geq 0}^{n-1} (-1)^{k} 
\binom{n+2}{k}  \binom{ n - 1 + i - k d}{n-1} = a_i,
\]
so 
\[b_i = \begin{cases}
1 & \text{if } i=0 \\
    \max(0,a_i) & \text{if $b_{i-1}$ is positive} \\
0 & \text{otherwise.} 
\end{cases}\]

Since $a_0=1$ and $a_i$ is positive for $i=1,\ldots,g(n,d)$ by Lemma \ref{lemma:differ}, we conclude that $b_i = a_i$ for
$i\leq g(n,d)$.

We will now prove that the two series differ in degree 
$g(n,d)+1$ for $n\geq 3, d \geq 2$ and $(n,d) \notin 
\{(3,2), (4,2), (6,2), (4,3) \}$. 
By Lemma \ref{lemma:differ}, we have $a_{g(n,d)+1} > 0$. If $b_{g(n,d)+1}=0$, then $a_{g(n,d)+1} > b_{g(n,d)+1}$, and we are done.  Otherwise $b_{g(n,d)+1} > 0$ and then \[a_{g(n,d)+1} - b_{g(n,d)+1}  = \begin{cases}
n & \text{ if } (n+2)(d-1) \equiv 0 \pmod{3}\\
1 & \text{ if }  (n+2)(d-1) \equiv 1 \pmod{3} \\
\binom{n+1}{2} & \text{ if } d\geq 3,  (n+2)(d-1) \equiv 2 \pmod{3} \\
\binom{n+1}{2} - n-2 & \text{ if } d = 2,  (n+2)(d-1) \equiv 2 \pmod{3}
\end{cases}
\]
by Lemma \ref{lemma:differ}.
\end{proof}

\begin{remark} Let 
\[ \sum_i b'_it^i =  \frac{(1-t^d)^{n+2}}{(1-t)^n}. \]
One can check that $b'_{g(n,d)} < 0$ for 
$(n,d) \in \{(7,2), (9,2), (5,3), (3,4), (3,5)\}$ and
$b'_{g(n,d)} = 0$ for $(n,d) \in \{(12,2),(3,3),(4,4)\}.$ We believe that $b'_{g(n,d)+1} > 0$ for the remaining values of $n$ and $d$. 
This would mean that  $R_{n,n+1,d}$ fails the WLP due injectivity in degree 
$g(n-1,d)$ 
except for the cases $(n,d) \in \{(8,2),(6,3),(4,4)\}$ where it 
fails in degree $g(n,d)$
due to surjectivity. This is consistent with \cite[Theorem 3.7]{BSV26} where it is shown that multiplication by a general form on $R_{n,n+1,2}$ is injective in degrees less than $g(n-1,2)$, and that it fails injectivity in degree $g(n-1,2).$ It is also consistent with \cite[Corollary 4.11]{BSV26} where it is shown that multiplication by a general form on $R_{n,n+1,3}$ is injective until degree $g(n-1,3)-1,$ and with \cite[Remark 4.12]{BSV26} where experimental evidence is given for the sharpness of this bound.
\end{remark}

\section{A connection to probability theory and the Iarrobino-Fr\"oberg conjecture}

In this section we will show that the Iarrobino-Fr\"oberg conjecture fails for all $n$ large enough for any given $d$ and $k$. In order to do this, we will need bounds for the degree of the Hilbert series predicted by Fr\"oberg and we will use
approximations from probability theory to achieve these bounds.  For $n$ and $m=n+k$, we would like to estimate when the coefficients of the polynomial
\[
  \frac{  (1-t^d)^{n+k}}{(1-t)^n} = (1+t+\cdots+t^{d-1})^n(1-t^d)^k = (1+t+\cdots+t^{d-1})^{n+k}(1-t)^k
\]
become negative for the first time. Observe that 
\[
 \frac{(1+t+\cdots+t^{d-1})^{n+k}}{d^{n+k}}
\]
is the generating function for the probability density for a sum of $n+k$ independent uniformly distributed random variables on
$\{0,1,\dots,d-1\}$. Hence we can use refinements of the central limit theorem from probability theory to estimate the coefficients for large $n$ and obtain an approximation of the degree of the Hilbert series predicted by
Fr\"oberg. 

For a uniformly distributed random variable $X$ in
$\{0,1,\dots,d-1\}$ we have that
\[
  \mathbb E(X) = \frac{d-1}{2}, \quad\sigma^2 = V(X) = \frac{d^2-1}{12}, \quad \text{and}\quad \mathbb
  E(X^k)<\infty, \quad\text{for all $k$}. 
\]

In the following theorem, $P_n(N)$ denotes the probability that $\sum_{i}^n X_i = N$, where $X_1,X_2,X_3,\dots,$ is a sequence of independent integer-valued random variables. If they are uniformly distributed on $\{0,1,\dots,d-1\}$, we have that $P_n(N)$ is the coefficient of $t^N$ in $(1+t+t^2+\cdots+t^{d-1})^n/d^n$.

\begin{definition}[{\cite[1.14 on page 139]{Petrov}}]
\label{dfn:qi}
For non-negative integers $i$, define
\[
q_i(x)
= \frac{1}{\sqrt{2\pi}} e^{-x^{2}/2}
   \sum H_{i + 2s}(x)
   \prod_{m=1}^{i}
      \frac{1}{k_m!}
      \left(
         \frac{\gamma_{\,m+2}}{(m+2)!\,\sigma^{\,m+2}}
      \right)^{k_m}.
\]
where the sum is taken over all non-negative solutions to the
equations $k_1+2k_2+\dots+ik_i=i$ and $k_1+k_2+\dots+k_i=s$. ($H_j$ is
the Hermite polynomial of degree $j$ and $\gamma_i$ is the
cumulant of order $i$.)
\end{definition}

\begin{theorem}[{\cite[Thm 13, Chapter VII]{Petrov}}]\label{thm:Petrov}
Let \(X_n\) be a sequence of integer-valued random variables having a common distribution. Suppose that \(V(X_1)= \sigma^2>0\), \(\mathbb E|X_1|^k<\infty\).
for some integer $k\ge 3$ and that the maximal span of the distribution $X_1$ is equal to $1$. Then
\[
\sigma \sqrt{n}\, P_n(N)
=
\frac{1}{\sqrt{2\pi}}\, e^{-x^{2}/2}
  + \sum_{i = 1}^{k-2} \frac{q_i(x)}{n^{i/2}}
  + o\!\left( \frac{1}{n^{(k-2)/2}} \right),
\]
uniformly in \(N\)\; \((-\infty < N < \infty)\).
Here \(x = \frac{N - n \mathbb E (X_1)}{\sigma \sqrt{n}}\)
and the \(q_i(x)\)  are the functions defined in Definition~\ref{dfn:qi}.
\end{theorem}

For functions $\mathbb N\rightarrow \mathbb Q$ we define the backward
difference operator $\Delta$ by $(\Delta f)(i) = f(i)-f(i-1)$, for all
$i\in \mathbb N$. For functions $\mathbb R\longrightarrow \mathbb R$
and a positive number $h$, we define the backward difference operator
$\nabla_h$ by $(\nabla_h f)(x) = f(x)-f(x-h)$. It is well known that
\[
  (\nabla_h^k f)(x)  = h^k f^{(k)}(x) + o(h^{k+1}).
\]
for sufficiently differentiable functions. Note that the functions
$q_i$, $i\ge 1$ and $e^{-x^2/2}$ are analytic.

\begin{theorem}\label{thm:Fdeg}
Let $D_{n,m,d}$ be the degree of the Hilbert series of $R_{n,m,d}$ and
let $\xi^{(k)}_1$ be the smallest root of the Hermite
polynomial $H_{k}(x)$. Then
\[
  \lim_{n\to\infty} \frac{D_{n,n+k,d}-(n+k)(d-1)/2}{\sqrt{(n+k)(d^2-1)/12}} = \xi^{(k)}_1.
\]
\end{theorem}

\begin{proof}
The coefficient of $x^N$ in $(1+x+\dots+x^{d-1})^{n+k}(1-x)^{k}$ is
given by
\[
  d^{n+k}\left(\Delta^{k} P_{n+k}\right)(N)
\]
By Theorem~\ref{thm:Petrov}, we have that
\[
  \sigma \sqrt{n+k}\, P_{n+k}(N)
=
\frac{1}{\sqrt{2\pi}}\, e^{-x^{2}/2}
  + \sum_{i = 1}^{m+1} \frac{q_i(x)}{\sqrt{n}^{i}}
  + o\!\left( \frac{1}{\sqrt{n}^{m+1}} \right),
\]
and with $h = 1/(\sigma\sqrt{n+k})$ and $q_0(x) = 1/\sqrt{2\pi}e^{-x^2/2}$, we get 
\begin{multline*}
  \frac{1}{h}\left(\Delta^k P_{n+k}\right)(N) = 
  \sum_{i = 0}^{k+1} (\sigma h)^i(\nabla_h^kq_i)(x)
  + o\!\left( h^{k+1} \right)\\
  = 
  \sum_{i = 0}^{k+1} \left(\sigma^i h^{k+i} q_i^{(k)}(x) +
    o(h^{k+1+i})\right)
  + o\!\left( h^{k+1} \right) = h^k q_0^{(k)} (x)+ o\!\left( h^{k+1} \right) 
\end{multline*}
This means that for $n$ large enough,
  $\left(\Delta^kP_{n+k}\right)(N)$ will have the same sign as 
\[
q_0^{(k)} (x) = q_0^{(k)}\left(\frac{N-(n+k)(d-1)/2}{\sigma\sqrt{n+k}}\right).
\]
If $\xi^{(k)}_1<\xi^{(k)}_2<\dots\xi^{(k)}_k,$ are roots of $q_0^{(k)}(x)=(-1)^{k}H(x)q_0(x)$, we have that
for $n$ large enough $\left(\Delta^k P_{n+k}\right)(N)>0$, for 
\[
  0\le N < (n+k)(d-1)/2+\xi^{(k)}_1 \sigma\sqrt{n+k}
\]
and $\left(\Delta^k P_{n+k}\right)(N)<0$, for 
\[
  \qquad (n+k)(d-1)/2+\xi^{(k)}_1 \sigma\sqrt{n+k}\le N <
  (n+k)(d-1)/2+\xi^{(k)}_2 \sigma\sqrt{n+k}.
\]
Since the degree of $R_{n,n+k,d}$ is given by the first sign change of 
\(\left(\Delta^m P_{n+k}\right)(N) \) we conclude that
\begin{equation} \label{eq:limit} 
  \lim_{n\to\infty}
  \frac{D_{n,n+k,d}-(n+k)(d-1)/2}{\sqrt{(n+k)(d^2-1)/12}} = \xi^{(k)}_1.
 \end{equation}
\end{proof}

\begin{remark}
    The convergence is faster if we use symmetric differences since $(\nabla^k_h f)(x) = h^kf^{(k)}(x-hk/2) + o(h^{k+2})$. It is also worth noting that $q_i(x)=0$ for odd $i$.
    \begin{figure}[!ht]
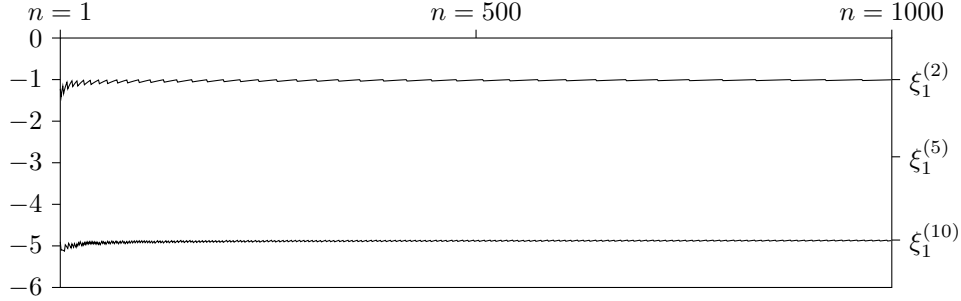

    \centering


    \caption{We illustrate the convergence in Equation~\ref{eq:limit} for $k=2$, $k=5$ and $k=10$ with $n$ from $1$ to $1000$ and $d=5$. $H_2(x) = x^2-1$ and $\xi^{(2)}_1=-1$, $H_5(x)=-x^{5}+10\,x^{3}-15\,x$ and $\xi^{(5)}_1 = -\sqrt{5+\sqrt{10}} \approx -2.857$, $H_{10}(x) = x^{10}-45\,x^{8}+630\,x^{6}-3150\,x^{4}+4725\,x^{2}-945$ and $\xi^{(10)}_1 \approx -4.859$}
    \label{fig:placeholder}
\end{figure}
\end{remark}

We can now state and prove the main result in this section which shows that the Iarrobino-Fr\"oberg Conjecture fails for all large enough values of $n$ for any fixed $d\ge 2$ and $k\ge 2$.

\begin{theorem} \label{thm:iarrobino}
For any $k\ge 2$ and any $d\ge 2$, the Hilbert series of $R_{n,n+k,d}$ differs from the Hilbert series predicted by Fröberg for all $n$ large enough. 
\end{theorem}

\begin{proof}
We have from Theorem~\ref{thm:diffdeg} that there are non-trivial syzygies among $n+2$ $d$'th powers of general linear forms in degree $g(n,d) + 1$, where \[
g(n,d) = \left \lfloor \frac{(n+1)(d-1)+1}{3} \right \rfloor.
\]
These syzygies remain for $n+k$ $d$th powers of general linear forms when $k\ge2$. Observe they are not in the module generated by the Koszul syzygies and they only involve $n+2$ of the generators.
By Theorem \ref{thm:Fdeg} this is below $D_{n,n+k,d}$ for $k\ge 2$, $d\ge 2$ and $n$ large enough since 
\begin{multline}\label{eq:compare}
\frac{g(n,d)+1-(n+k)(d-1)/2}{\sqrt{(n+k)(d^2-1)/12}} <\frac{(n+1)(d-1)/3+2 -(n+k)(d-1)/2}{\sqrt{(n+k)(d^2-1)/12}} \\
<\frac{(n+k)(d-1)/3+2 -(n+k)(d-1)/2}{\sqrt{(n+k)(d^2-1)/12}} = \frac{2-(n+k)(d-1)/6}{\sqrt{(n+k)(d^2-1)/12}}
\end{multline}
which tends to $-\infty$ when $n$ grows. Hence the Hilbert series of $R_{n,n+k,d}$ does not agree with the Hilbert series predicted by Fröberg.
\end{proof}

\begin{remark}
    Since Theorem~\ref{thm:iarrobino} does not give precise information about what large enough means, it is useful to look at some examples as illustrated in Figure~\ref{fig:iarrobino3} for $k=2$, $k=5$ and $k=10$, respectively. We also plot the graph for the quotient in Equation~\ref{eq:compare} compared to the quotient in Equation~\ref{eq:limit} for $d=5$ and $n$ from $1$ to $100$ in Figure~\ref{fig:graf}.
\end{remark}
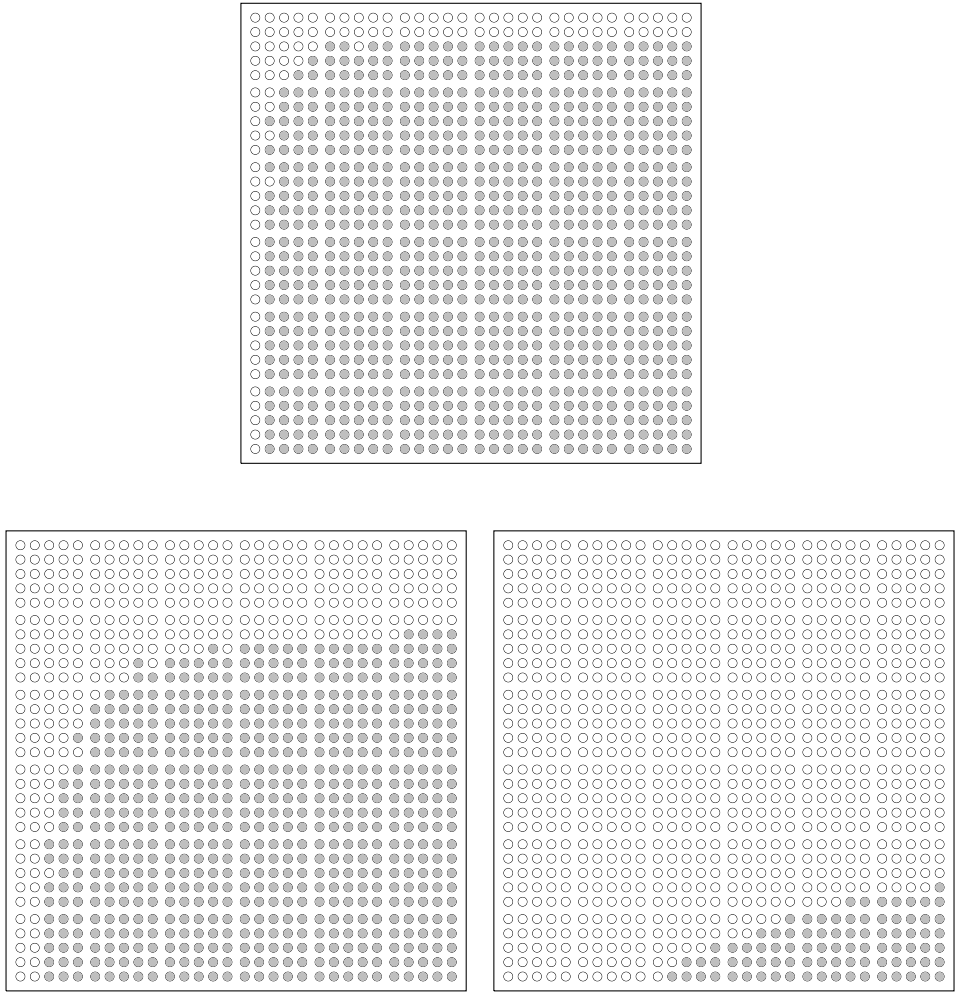
\begin{figure}[!ht]
\centering
\begin{tikzpicture}[scale=0.19]
 \draw (0,0) -- (32,0) -- (32,32) -- (0,32) -- (0,0);
  \def\gap{0.20}
  \newcommand{\coord}[2]{
    \pgfmathtruncatemacro{\bx}{(#1-1)/5}
    \pgfmathtruncatemacro{\by}{(#2-1)/5}
    \pgfmathsetmacro{\xx}{#1 + \gap*\bx}
    \pgfmathsetmacro{\yy}{#2 + \gap*\by}
  }

  \newcommand{\one}[2]{
    \coord{#1}{#2}
    \fill[gray!50] (\xx,\yy) circle (0.3);
  }

  \newcommand{\zero}[2]{
    \coord{#1}{#2}
    \draw[line width=0.18pt] (\xx,\yy) circle (0.3);
    \fill[white] (\xx,\yy) circle (0.3);
  }

  \foreach \x in {1,...,30}{
    \foreach \y in {1,...,30}{
      \zero{\x}{\y}
    }
  }

  \foreach \x in {6,7}{ \one{\x}{28} }
  \foreach \x in {9,...,30}{ \one{\x}{28} }
  \foreach \x in {5,...,30}{ \one{\x}{27} }
  \foreach \x in {4,...,30}{ \one{\x}{26} }
  \foreach \x in {3,...,30}{ \one{\x}{25} }
  \foreach \x in {3,...,30}{ \one{\x}{24} }
  \foreach \x in {2,...,30}{ \one{\x}{23} }
  \foreach \x in {3,...,30}{ \one{\x}{22} }
  \foreach \x in {2,...,30}{ \one{\x}{21} }
  \foreach \x in {2,...,30}{ \one{\x}{20} }
  \foreach \x in {3,...,30}{ \one{\x}{19} }
  \foreach \y in {1,...,18}{
    \foreach \x in {2,...,30}{ \one{\x}{\y} }
  }
\end{tikzpicture}

\begin{tabular}{cc}
\\
\\
\begin{tikzpicture}[scale=0.19]
 \draw (0,0) -- (32,0) -- (32,32) -- (0,32) -- (0,0);
  \def\gap{0.20}

  \newcommand{\coord}[2]{
    \pgfmathtruncatemacro{\bx}{(#1-1)/5}
    \pgfmathtruncatemacro{\by}{(#2-1)/5}
    \pgfmathsetmacro{\xx}{#1 + \gap*\bx}
    \pgfmathsetmacro{\yy}{#2 + \gap*\by}
  }

  \newcommand{\one}[2]{
    \coord{#1}{#2}
    \fill[gray!50] (\xx,\yy) circle (0.3);
  }

  \newcommand{\zero}[2]{
    \coord{#1}{#2}
    \draw[line width=0.18pt] (\xx,\yy) circle (0.3);
    \fill[white] (\xx,\yy) circle (0.3);
  }

  \foreach \x in {1,...,30}{
    \foreach \y in {1,...,30}{
      \zero{\x}{\y}
    }
  }

  \foreach \x in {27,28,29,30}{ \one{\x}{24} }

  \one{14}{23}
  \foreach \x in {16,...,30}{ \one{\x}{23} }
  \one{9}{22}
  \foreach \x in {11,...,30}{ \one{\x}{22} }
  \foreach \x in {9,...,30}{ \one{\x}{21} }
  \foreach \x in {7,...,30}{ \one{\x}{20} }
  \foreach \x in {6,...,30}{ \one{\x}{19} }
  \foreach \x in {6,...,30}{ \one{\x}{18} }
  \foreach \x in {5,...,30}{ \one{\x}{17} }
  \foreach \x in {6,...,30}{ \one{\x}{16} }
  \foreach \x in {5,...,30}{ \one{\x}{15} }
  \foreach \y in {11,...,14}{
    \foreach \x in {4,...,30}{ \one{\x}{\y} }
  }
  \foreach \y in {9,10}{
    \foreach \x in {3,...,30}{ \one{\x}{\y} }
  }
  \foreach \x in {4,...,30}{ \one{\x}{8} }
  \foreach \y in {1,...,7}{
    \foreach \x in {3,...,30}{ \one{\x}{\y} }
  }
\end{tikzpicture}
&\begin{tikzpicture}[scale=0.19]
 \draw (0,0) -- (32,0) -- (32,32) -- (0,32) -- (0,0);
  \def\gap{0.20}
  \newcommand{\coord}[2]{
    \pgfmathtruncatemacro{\bx}{(#1-1)/5}
    \pgfmathtruncatemacro{\by}{(#2-1)/5}
    \pgfmathsetmacro{\xx}{#1 + \gap*\bx}
    \pgfmathsetmacro{\yy}{#2 + \gap*\by}
  }
  \newcommand{\one}[2]{
    \coord{#1}{#2}
    \fill[gray!50] (\xx,\yy) circle (0.3);
  }
  \newcommand{\zero}[2]{
    \coord{#1}{#2}
    \draw[line width=0.18pt] (\xx,\yy) circle (0.3);
    \fill[white] (\xx,\yy) circle (0.3);
  }
  \foreach \x in {1,...,30}{
    \foreach \y in {1,...,30}{
      \zero{\x}{\y}
    }
  }
  \one{30}{7}
  \foreach \x in {24,...,30}{ \one{\x}{6} }
  \foreach \x in {20,...,30}{ \one{\x}{5} }
  \foreach \x in {18,...,30}{ \one{\x}{4} }
  \foreach \x in {15,...,30}{ \one{\x}{3} }
  \foreach \x in {13,...,30}{ \one{\x}{2} }
  \foreach \x in {12,...,30}{ \one{\x}{1} }
\end{tikzpicture}
\\
\end{tabular}
\caption{
These diagrams illustrate for which pairs $(n,d)$ there are non-trivial syzygies in degrees less than or equal to the Fr\"oberg, where the number of general forms is $n+2$, $n+5$ and $n+10$, respectively. A gray circle in position $(n,d)$ means that we know that such syzygies exists while a white circle means that they don't necesarily exist. Thus the gray circles correspond to counterexamples to the Iarrobino-Fr\"oberg conjecture. The top left corner correspond to $(n,d)=(1,1)$, $d$ increases to the right and $n$ increases downwards. Chandler's results show that rows $5-7$ in the second case are eventually all gray. Our result shows that all columns except the first are eventually all gray. \vspace{1cm}
}
\label{fig:iarrobino3}
\end{figure}

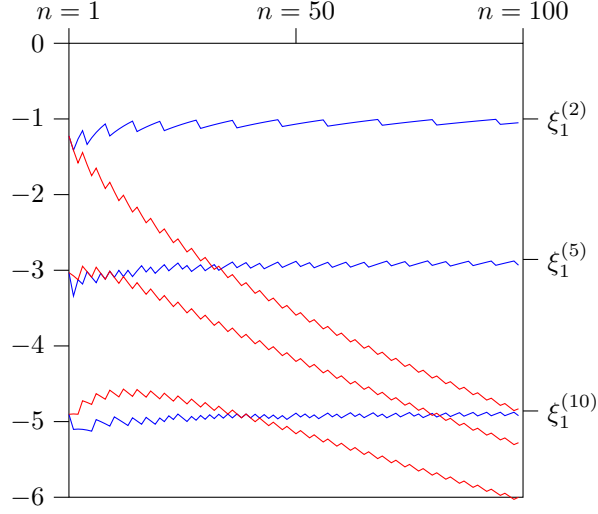
\begin{figure}
    \centering

\begin{tikzpicture}[scale=0.20]
 \draw (0,0) -- (30,0) -- (30,-30) -- (0,-30) -- (0,0);
 \draw (0,0) -- (0,1) node[above] {$n=1$};
 \draw (30,0) -- (30,1) node[above] {$n=100$};
 \draw (15,0) -- (15,1) node[above] {$n=50$};
 \draw (0,0) -- (-1,0) node[left]{$0$};
 \draw (0,-5) -- (-1,-5) node[left]{$-1$}; 
 \draw (0,-10) -- (-1,-10) node[left]{$-2$}; 
 \draw (0,-15) -- (-1,-15) node[left]{$-3$}; 
 \draw (0,-20) -- (-1,-20) node[left]{$-4$}; 
 \draw (0,-25) -- (-1,-25) node[left]{$-5$}; 
 \draw (0,-30) -- (-1,-30) node[left]{$-6$};
 \draw (30,-5) -- (31,-5) node[right]{$\xi^{(2)}_1$};
 \draw (30,-14.2848) -- (31,-14.2848) node[right]{$\xi^{(5)}_1$};
 \draw (30,-24.29731) -- (31,-24.29731) node[right]{$\xi^{(10)}_1$};

 \draw[color=blue] (0,-6.12372) --(.3,-7.07107) --(.6,-6.32456) --(.9,-5.7735) --(1.2,-6.68153) --(1.5,-6.25) --(1.8,-5.89256) --(2.1,-5.59017) --(2.4,-5.33002) --(2.7,-6.12372) --(3,-5.88348) --(3.3,-5.66947) --(3.6,-5.47723) --(3.9,-5.3033) --(4.2,-5.14496) --(4.5,-5.83333) --(4.8,-5.67775) --(5.1,-5.53399) --(5.4,-5.40062) --(5.7,-5.27645) --(6,-5.16047) --(6.3,-5.7735) --(6.6,-5.65685) --(6.9,-5.547) --(7.2,-5.44331) --(7.5,-5.34522) --(7.8,-5.25226) --(8.1,-5.16398) --(8.4,-5.08001) --(8.7,-5.625) --(9,-5.53912) --(9.3,-5.45705) --(9.6,-5.37853) --(9.9,-5.3033) --(10.2,-5.23114) --(10.5,-5.16185) --(10.8,-5.09525) --(11.1,-5.59017) --(11.4,-5.52158) --(11.7,-5.45545) --(12,-5.39164) --(12.3,-5.33002) --(12.6,-5.27046) --(12.9,-5.21286) --(13.2,-5.15711) --(13.5,-5.1031) --(13.8,-5.05076) --(14.1,-5.5) --(14.4,-5.44581) --(14.7,-5.39319) --(15,-5.34207) --(15.3,-5.29238) --(15.6,-5.24404) --(15.9,-5.19701) --(16.2,-5.15122) --(16.5,-5.10662) --(16.8,-5.06316) --(17.1,-5.47723) --(17.4,-5.43214) --(17.7,-5.38816) --(18,-5.34522) --(18.3,-5.3033) --(18.6,-5.26235) --(18.9,-5.22233) --(19.2,-5.18321) --(19.5,-5.14496) --(19.8,-5.10754) --(20.1,-5.07093) --(20.4,-5.03509) --(20.7,-5.41667) --(21,-5.37944) --(21.3,-5.34297) --(21.6,-5.30723) --(21.9,-5.2722) --(22.2,-5.23785) --(22.5,-5.20416) --(22.8,-5.17112) --(23.1,-5.1387) --(23.4,-5.10688) --(23.7,-5.07565) --(24,-5.04498) --(24.3,-5.40062) --(24.6,-5.36875) --(24.9,-5.33745) --(25.2,-5.30669) --(25.5,-5.27645) --(25.8,-5.24672) --(26.1,-5.21749) --(26.4,-5.18875) --(26.7,-5.16047) --(27,-5.13265) --(27.3,-5.10527) --(27.6,-5.07833) --(27.9,-5.05181) --(28.2,-5.02571) --(28.5,-5.35714) --(28.8,-5.33002) --(29.1,-5.3033) --(29.4,-5.27698) --(29.7,-5.25105) ;
\draw[color=blue] (0,-15.1554) --(.3,-16.7038) --(.6,-15.625) --(.9,-15.9099) --(1.2,-15.0935) --(1.5,-15.4571) --(1.8,-15.8196) --(2.1,-15.199) --(2.4,-15.591) --(2.7,-15.0624) --(3,-15.468) --(3.3,-15.0061) --(3.6,-15.4167) --(3.9,-15.0055) --(4.2,-15.4161) --(4.5,-15.0446) --(4.8,-14.6987) --(5.1,-15.1128) --(5.4,-14.7946) --(5.7,-15.2028) --(6,-14.9076) --(6.3,-14.6289) --(6.6,-15.0334) --(6.9,-14.772) --(7.2,-14.5237) --(7.5,-14.9225) --(7.8,-14.6875) --(8.1,-15.0787) --(8.4,-14.8553) --(8.7,-14.6416) --(9,-15.026) --(9.3,-14.8216) --(9.6,-14.6253) --(9.9,-15.0027) --(10.2,-14.814) --(10.5,-14.6322) --(10.8,-14.4569) --(11.1,-14.827) --(11.4,-14.6575) --(11.7,-14.4938) --(12,-14.8567) --(12.3,-14.6978) --(12.6,-14.5438) --(12.9,-14.8998) --(13.2,-14.75) --(13.5,-14.6047) --(13.8,-14.4636) --(14.1,-14.8121) --(14.4,-14.6743) --(14.7,-14.5403) --(15,-14.4099) --(15.3,-14.7512) --(15.6,-14.6235) --(15.9,-14.4991) --(16.2,-14.8342) --(16.5,-14.7121) --(16.8,-14.5929) --(17.1,-14.4767) --(17.4,-14.805) --(17.7,-14.6907) --(18,-14.579) --(18.3,-14.4698) --(18.6,-14.7918) --(18.9,-14.6842) --(19.2,-14.5789) --(19.5,-14.4759) --(19.8,-14.7917) --(20.1,-14.69) --(20.4,-14.5904) --(20.7,-14.4928) --(21,-14.3972) --(21.3,-14.7063) --(21.6,-14.6117) --(21.9,-14.5189) --(22.2,-14.4279) --(22.5,-14.7314) --(22.8,-14.6413) --(23.1,-14.5528) --(23.4,-14.4659) --(23.7,-14.3806) --(24,-14.678) --(24.3,-14.5934) --(24.6,-14.5102) --(24.9,-14.4285) --(25.2,-14.7208) --(25.5,-14.6397) --(25.8,-14.5599) --(26.1,-14.4814) --(26.4,-14.4042) --(26.7,-14.6909) --(27,-14.6142) --(27.3,-14.5387) --(27.6,-14.4643) --(27.9,-14.391) --(28.2,-14.6725) --(28.5,-14.5996) --(28.8,-14.5279) --(29.1,-14.4572) --(29.4,-14.3875) --(29.7,-14.6639) ;
\draw[color=blue] (0,-24.5181) --(.3,-25.5155) --(.6,-25.4951) --(.9,-25.5126) --(1.2,-25.5604) --(1.5,-25.6326) --(1.8,-24.8673) --(2.1,-25) --(2.4,-25.1443) --(2.7,-25.2982) --(3,-24.6885) --(3.3,-24.8747) --(3.6,-25.0651) --(3.9,-25.2591) --(4.2,-24.7487) --(4.5,-24.9615) --(4.8,-25.1753) --(5.1,-24.7217) --(5.4,-24.9482) --(5.7,-25.1744) --(6,-24.765) --(6.3,-25) --(6.6,-24.6183) --(6.9,-24.8599) --(7.2,-24.5022) --(7.5,-24.7487) --(7.8,-24.9932) --(8.1,-24.6622) --(8.4,-24.9101) --(8.7,-24.5967) --(9,-24.8471) --(9.3,-24.5495) --(9.6,-24.8015) --(9.9,-24.5181) --(10.2,-24.7712) --(10.5,-24.5004) --(10.8,-24.7541) --(11.1,-24.4949) --(11.4,-24.7487) --(11.7,-24.5) --(12,-24.7537) --(12.3,-24.5145) --(12.6,-24.7678) --(12.9,-24.5374) --(13.2,-24.79) --(13.5,-24.5677) --(13.8,-24.8195) --(14.1,-24.6046) --(14.4,-24.8555) --(14.7,-24.6475) --(15,-24.4447) --(15.3,-24.6957) --(15.6,-24.4989) --(15.9,-24.7487) --(16.2,-24.5576) --(16.5,-24.8061) --(16.8,-24.6203) --(17.1,-24.4385) --(17.4,-24.6864) --(17.7,-24.5095) --(18,-24.7559) --(18.3,-24.5833) --(18.6,-24.4144) --(18.9,-24.6598) --(19.2,-24.4949) --(19.5,-24.7388) --(19.8,-24.5776) --(20.1,-24.4195) --(20.4,-24.6623) --(20.7,-24.5077) --(21,-24.7487) --(21.3,-24.5974) --(21.6,-24.4487) --(21.9,-24.6885) --(22.2,-24.5429) --(22.5,-24.3998) --(22.8,-24.6382) --(23.1,-24.4978) --(23.4,-24.7345) --(23.7,-24.5967) --(24,-24.4612) --(24.3,-24.6965) --(24.6,-24.5634) --(24.9,-24.4324) --(25.2,-24.6662) --(25.5,-24.5374) --(25.8,-24.4106) --(26.1,-24.6429) --(26.4,-24.5181) --(26.7,-24.3952) --(27,-24.6259) --(27.3,-24.5049) --(27.6,-24.3857) --(27.9,-24.6148) --(28.2,-24.4973) --(28.5,-24.3815) --(28.8,-24.6091) --(29.1,-24.4949) --(29.4,-24.3823) --(29.7,-24.6083) ;
\draw[color=red] (0,-6.12372) --(.3,-7.07107) --(.6,-7.90569) --(.9,-7.21688) --(1.2,-8.01784) --(1.5,-8.75) --(1.8,-8.24958) --(2.1,-8.94427) --(2.4,-9.59403) --(2.7,-9.18559) --(3,-9.80581) --(3.3,-10.394) --(3.6,-10.0416) --(3.9,-10.6066) --(4.2,-11.1474) --(4.5,-10.8333) --(4.8,-11.3555) --(5.1,-11.8585) --(5.4,-11.5728) --(5.7,-12.0605) --(6,-12.5326) --(6.3,-12.2687) --(6.6,-12.7279) --(6.9,-13.1741) --(7.2,-12.9279) --(7.5,-13.3631) --(7.8,-13.7872) --(8.1,-13.5554) --(8.4,-13.97) --(8.7,-14.375) --(9,-14.1555) --(9.3,-14.5521) --(9.6,-14.9404) --(9.9,-14.7314) --(10.2,-15.1122) --(10.5,-15.4856) --(10.8,-15.2857) --(11.1,-15.6525) --(11.4,-16.0126) --(11.7,-15.8208) --(12,-16.1749) --(12.3,-16.5231) --(12.6,-16.3384) --(12.9,-16.6812) --(13.2,-17.0185) --(13.5,-16.8402) --(13.8,-17.1726) --(14.1,-17.5) --(14.4,-17.3276) --(14.7,-17.6505) --(15,-17.9688) --(15.3,-17.8016) --(15.6,-18.1158) --(15.9,-18.4258) --(16.2,-18.2634) --(16.5,-18.5695) --(16.8,-18.8718) --(17.1,-18.7139) --(17.4,-19.0125) --(17.7,-19.3076) --(18,-19.1537) --(18.3,-19.4454) --(18.6,-19.7338) --(18.9,-19.5837) --(19.2,-19.869) --(19.5,-20.1511) --(19.8,-20.0045) --(20.1,-20.2837) --(20.4,-20.5599) --(20.7,-20.4167) --(21,-20.6901) --(21.3,-20.9609) --(21.6,-20.8207) --(21.9,-21.0888) --(22.2,-21.3543) --(22.5,-21.217) --(22.8,-21.48) --(23.1,-21.7407) --(23.4,-21.606) --(23.7,-21.8643) --(24,-22.1203) --(24.3,-21.9882) --(24.6,-22.242) --(24.9,-22.4935) --(25.2,-22.3639) --(25.5,-22.6134) --(25.8,-22.8607) --(26.1,-22.7334) --(26.4,-22.9787) --(26.7,-23.2221) --(27,-23.0969) --(27.3,-23.3384) --(27.6,-23.578) --(27.9,-23.4549) --(28.2,-23.6926) --(28.5,-23.9286) --(28.8,-23.8074) --(29.1,-24.0416) --(29.4,-24.2741) --(29.7,-24.1548) ;
\draw[color=red] (0,-15.1554) --(.3,-15.3675) --(.6,-15.625) --(.9,-14.7314) --(1.2,-15.0935) --(1.5,-15.4571) --(1.8,-14.799) --(2.1,-15.199) --(2.4,-15.591) --(2.7,-15.0624) --(3,-15.468) --(3.3,-15.8636) --(3.6,-15.4167) --(3.9,-15.8166) --(4.2,-16.2067) --(4.5,-15.8161) --(4.8,-16.2062) --(5.1,-16.5872) --(5.4,-16.238) --(5.7,-16.617) --(6,-16.9877) --(6.3,-16.6701) --(6.6,-17.0379) --(6.9,-17.3981) --(7.2,-17.1057) --(7.5,-17.4625) --(7.8,-17.8125) --(8.1,-17.5405) --(8.4,-17.887) --(8.7,-18.2272) --(9,-17.9723) --(9.3,-18.309) --(9.6,-18.64) --(9.9,-18.3995) --(10.2,-18.7271) --(10.5,-19.0494) --(10.8,-18.8213) --(11.1,-19.1403) --(11.4,-19.4546) --(11.7,-19.2372) --(12,-19.5482) --(12.3,-19.8549) --(12.6,-19.6469) --(12.9,-19.9505) --(13.2,-20.25) --(13.5,-20.0505) --(13.8,-20.347) --(14.1,-20.6398) --(14.4,-20.4478) --(14.7,-20.7378) --(15,-21.0243) --(15.3,-20.839) --(15.6,-21.1228) --(15.9,-21.4034) --(16.2,-21.2242) --(16.5,-21.5022) --(16.8,-21.7771) --(17.1,-21.6036) --(17.4,-21.8761) --(17.7,-22.1457) --(18,-21.9773) --(18.3,-22.2446) --(18.6,-22.5092) --(18.9,-22.3455) --(19.2,-22.6079) --(19.5,-22.8677) --(19.8,-22.7083) --(20.1,-22.9661) --(20.4,-23.2214) --(20.7,-23.066) --(21,-23.3193) --(21.3,-23.5703) --(21.6,-23.4187) --(21.9,-23.6678) --(22.2,-23.9147) --(22.5,-23.7666) --(22.8,-24.0117) --(23.1,-24.2547) --(23.4,-24.1099) --(23.7,-24.3511) --(24,-24.5904) --(24.3,-24.4487) --(24.6,-24.6862) --(24.9,-24.9219) --(25.2,-24.7831) --(25.5,-25.0172) --(25.8,-25.2494) --(26.1,-25.1133) --(26.4,-25.344) --(26.7,-25.573) --(27,-25.4395) --(27.3,-25.667) --(27.6,-25.8929) --(27.9,-25.7618) --(28.2,-25.9862) --(28.5,-26.209) --(28.8,-26.0802) --(29.1,-26.3017) --(29.4,-26.5216) --(29.7,-26.395) ;
\draw[color=red] (0,-24.5181) --(.3,-24.4949) --(.6,-24.5145) --(.9,-23.6228) --(1.2,-23.7346) --(1.5,-23.8649) --(1.8,-23.1523) --(2.1,-23.3333) --(2.4,-23.5221) --(2.7,-22.9265) --(3,-23.1455) --(3.3,-23.3671) --(3.6,-22.8535) --(3.9,-23.094) --(4.2,-23.3345) --(4.5,-22.8814) --(4.8,-23.1341) --(5.1,-23.3854) --(5.4,-22.9786) --(5.7,-23.2379) --(6,-23.495) --(6.3,-23.125) --(6.6,-23.3874) --(6.9,-23.6472) --(7.2,-23.307) --(7.5,-23.5702) --(7.8,-23.8308) --(8.1,-23.5151) --(8.4,-23.7778) --(8.7,-24.0377) --(9,-23.7428) --(9.3,-24.004) --(9.6,-24.2624) --(9.9,-23.9851) --(10.2,-24.2441) --(10.5,-24.5004) --(10.8,-24.2384) --(11.1,-24.4949) --(11.4,-24.7487) --(11.7,-24.5) --(12,-24.7537) --(12.3,-25.0048) --(12.6,-24.7678) --(12.9,-25.0185) --(13.2,-25.2668) --(13.5,-25.0401) --(13.8,-25.2878) --(14.1,-25.5331) --(14.4,-25.3158) --(14.7,-25.5604) --(15,-25.8027) --(15.3,-25.5938) --(15.6,-25.8353) --(15.9,-26.0746) --(16.2,-25.8732) --(16.5,-26.1116) --(16.8,-26.348) --(17.1,-26.1535) --(17.4,-26.389) --(17.7,-26.6224) --(18,-26.4342) --(18.3,-26.6667) --(18.6,-26.8972) --(18.9,-26.7148) --(19.2,-26.9444) --(19.5,-27.1721) --(19.8,-26.9951) --(20.1,-27.2218) --(20.4,-27.4467) --(20.7,-27.2746) --(21,-27.4986) --(21.3,-27.7208) --(21.6,-27.5533) --(21.9,-27.7746) --(22.2,-27.9942) --(22.5,-27.831) --(22.8,-28.0496) --(23.1,-28.2667) --(23.4,-28.1074) --(23.7,-28.3235) --(24,-28.5381) --(24.3,-28.3826) --(24.6,-28.5962) --(24.9,-28.8083) --(25.2,-28.6563) --(25.5,-28.8675) --(25.8,-29.0773) --(26.1,-28.9286) --(26.4,-29.1374) --(26.7,-29.3449) --(27,-29.1993) --(27.3,-29.4059) --(27.6,-29.6112) --(27.9,-29.4684) --(28.2,-29.6728) --(28.5,-29.8759) --(28.8,-29.736) --(29.1,-29.9382) --(29.4,-30.1392) --(29.7,-30.0019) ;
\end{tikzpicture}

    \caption{We illustrate the difference between the two sides of the inequality in Equation~\ref{eq:limit} for $d=5$ and $k=2$, $k=5$ and $k=10$. The left hand side is colored red and the right hand side is colored blue.}
    \label{fig:graf}
\end{figure}

\vspace{1cm}

\section{Two generic quadratic forms in the exterior algebra}

Let $E_n$ denote the exterior algebra on $n$ generators and let $f$ and $g$ be two generic quadratic forms in $E_n$.  Let $c(n, s)$ denote the number of lattice paths inside the rectangle $(n+2-2s) \times (n+2)$ from the bottom left corner to the top right corner with moves of two types: $(x,y) \mapsto (x+1,y+1)$ or $(x-1,y+1)$. 

In \cite{CLN19} two conjectures are given which relates the cofficients $c(n,s)$ to the dimensions of the graded pieces of $E_n/(f,g)$, and $R_{n,n+2,2}$.

\begin{conjecture} \label{conj:exterior}
    
The Hilbert series of $E_n/(f,g)$ is equal to $1+c(n,1)t+c(n,2)t^2 + \cdots + c(n, \lfloor \frac{n}{2} \rfloor) t^{\lfloor \frac{n}{2} \rfloor}$.

\end{conjecture}
\begin{conjecture} \label{conj:squares}
The Hilbert series of $R_{n,n+2,2}$ is equal to $1+c(n,1)t+c(n,2)t^2 + \cdots + c(n, \lfloor \frac{n}{2} \rfloor) t^{\lfloor \frac{n}{2} \rfloor}$.
\end{conjecture}
It is shown, see \cite[Theorem 1]{CLN19}, that if Conjecture \ref{conj:exterior} is true for all $n$,  then so is Conjecture \ref{conj:squares}, and for the opposite direction, if Conjecture \ref{conj:squares} is true for all $n$, then  
Conjecture \ref{conj:exterior} is true for even $n$.

In this section we prove Conjecture \ref{conj:squares}. Our argument rely on two technical lemmas. The first of these two lemmas is illustrated by Figure \ref{fig:d2lemma1}.

\begin{figure}[ht]
\centering
\begin{tikzpicture}[scale=0.6]
  \draw[step=1cm,gray!30,very thin] (-2.5,-0.5) grid (10.5,12.5);
  \draw[very thick] (-2.5,0) -- (10.5,0);      
  \draw[very thick] (0,-0.5) -- (0,12.5);      

  \draw[very thick,blue!60] (2,-0.5) -- (2,12.5);

  \draw[very thick,blue!60,dashed] (10,-0.5) -- (10,12.5);

  \draw[very thick,red!70] (5,-0.5) -- (5,12.5);

  \draw[very thick,red!70,dashed] (7,-0.5) -- (7,12.5);

  \fill (4,0) circle (2pt); 
  \fill (0,12) circle (2pt); 

  \draw[very thick,green!60!black]
    (4,0) -- (5,1) -- (4,2) -- (3,3) -- (2,4) -- (1,5)
    -- (0,6) -- (-1,7) -- (0,8) -- (1,9)
    -- (0,10) -- (-1,11) -- (0,12);

  \draw[very thick,red!70,dotted]
    (4,0) -- (3,1) -- (2,2)
    -- (3,3) -- (4,4) -- (5,5) -- (6,6)
    -- (5,7) -- (4,8) -- (3,9) -- (2,10) -- (1,11) -- (0,12);

  \fill (8,0) circle (2pt); 
  \draw[very thick,red!80!black,dash pattern=on 5pt off 3pt]
    (8,0) -- (9,1) -- (10,2) -- (9,3) -- (8,4) -- (7,5) -- (6,6)
    -- (5,7) -- (4,8) -- (3,9) -- (2,10) -- (1,11) -- (0,12);

\end{tikzpicture}
\caption{Illustration of the construction in Lemma \ref{lemma:rec} with $t=4, n= 12, i=2,j=5$.
The thick green path is valid and consists of $(n-t)/2 =4$ moves to
the right.
The dotted red path $p$ is invalid because it crosses the line $x=j$
before the line $x=i$. 
The dashed red path is the reflected path $p'$, obtained by reflecting the initial segment of $p$ in the line $x=j+1=6$. 
The dashed blue line at $x=10$ and the dashed red line at $x=7$ show the images of the blue line at $x=2$ and the red line at $x=5$ under this reflection.}
 \label{fig:d2lemma1}
\end{figure}
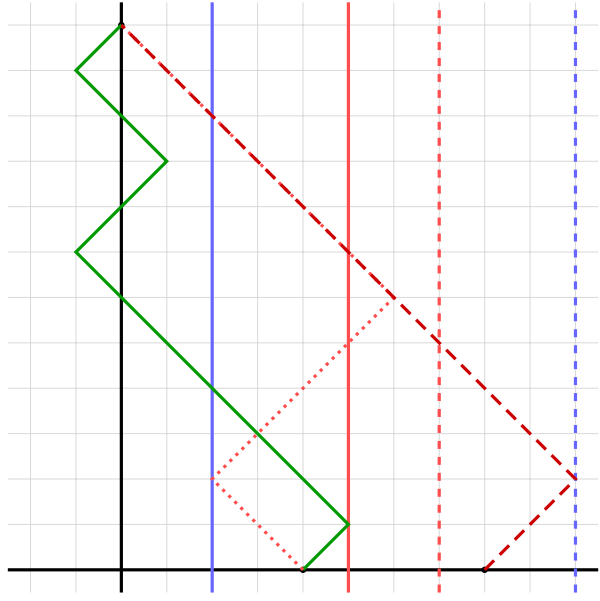

\begin{lemma} \label{lemma:rec}
Let $n,t,i,j$ be positive integers such that $n$ and $t$ have the same parity, and such that $i<t<j$.
Let $e(n,t,i,j)$ denote the number of lattice paths from the point $(t,0)$ to the point $(0,n)$ with moves of two types: 
$(x,y) \mapsto (x+1,y+1)$ or $(x-1,y+1)$, and with the restriction that the line $x=j$ must not be crossed before the line $x=i$ has been crossed. Then 
\[e(n,t,i,j) = \sum_{k \geq 0} \left( \binom{n}{(n-t)/2 - k (j-i+2)} - \binom{n}{(n+t)/2 - j - 1 - k(j-i+2)} \right).\]
\end{lemma}
\begin{proof}
We begin by showing that $e(n,t,i,j)$ satisfies the recursion
\begin{equation} \label{eq:rec}
e(n,t,i,j) = \binom{n}{(n-t)/2} - e(n,2(j+1)-t,j+2,2(j+1)-i).
\end{equation}

 Here we can observe that $n,2(j+1)-t,j+2,2(j+1)-i$ are positive integers, that $n$ and $2j-t$ have the same parity, and that $j+2<2(j+1)-t<2(j+1)-i$. 

A path from point $(t,0)$ to point $(0,n)$ will consist of $(n-t)/2$ right steps, so the number of paths without restriction is $\binom{n}{(n-t)/2}$. The forbidden paths are those that cross the line $x=j$ before the line $x=i$. 

Let $p$ be such a path and let $(j+1,b_p)$ be the position where $p$ has crossed the line $x=j$ for the first time. Let $p'$ be the path 
obtained from $p$ by first reflecting $p$ from $(t,0)$ to $(j+1,b_p)$ along the line $x=j+1$, and from the point $(j+1,b_p)$, letting it coincide with $p$. Then $p'$ is a path from $(2(j+1)-t,0)$ to $(0,n)$. Conversely, a path $p'$ from $(2(j+1)-t,0)$ to $(0,n)$ gives rise to a path $p$ from $(t,0)$ to $(0,n)$, and the latter crosses the line $x=j$ before the line $x=i$ if and only if $p'$ crosses the line $x=j+2$ before possibly crossing the line $x= 2(j+1)-i$. Thus, the number of forbidden paths equals $e(n,2(j+1)-t,j,2(j+1)-i)$, which shows that Equation \ref{eq:rec} holds.

Let $d = j-i+2$. The closed formula now follows by the observation that 
\begin{align*}
e(n,t,i,j) = \binom{n}{(n-t)/2} - e(n,2(j+1)-t,i+d,j+d) = \\
\binom{n}{(n-t)/2} - \binom{n}{(n+t)/2 - j-1} + \\e(n, 2(j+d+1) -( 2(j+1)-t),i+2d,j+2d) = \\
\binom{n}{(n-t)/2} - \binom{n}{(n+t)/2 - j-1} + e(n, t+ 2d,i+2d,j+2d). 
\end{align*}

\end{proof}

Let $c(n,s)$ the number of lattice paths inside the rectangle $(n+2-2s) \times (n+2)$ from the bottom left corner to the top right corner with moves of two types: $(x,y) \mapsto (x+1,y+1)$ or $(x-1,y+1)$.

\begin{lemma} \label{lemma:lattice}
Let $n-2s+2 \geq 1$ so that the rectangle is non-degenarate.
Then the number $c(n,s)$ equals 
\[ \sum_{k = 0}^{\lceil s/3 \rceil }  \left( \binom{n}{s-kd}  - 2\binom{n}{s-2-kd} + \binom{n}{s-4-kd} \right),\]
where $d=n-2s+4.$
\end{lemma}
\begin{proof}

The argument in the beginning of the proof of \cite[Proposition 7]{CLN19} gives that the number $c(n,s)$ equals 
\[\binom{n}{s} - 2\binom{n}{s-2} + u,\] where $u$ is the number of paths that cross both the line 
$x=0$ and $x=n+2-2s$.

Consider first the number of paths from $(1,1)$ to $(n+1-2s,n+1)$ that cross the line $x=n+2-2s$ before the line $x=0$. Let $(n+3-2s,b_p)$ be the first point on $p$ to the right of the line $x=n+2-2s$ and let $(-1,a_p)$ be the first point on $p$ to the left of the line $x=0$. We construct a new path $p'$ by reflecting the path from $(n+3-2s,b_p)$ to $(-1,a_p)$ along the line $x=0$, while for points before $(n+3-2s,b_p)$ and after $(-1,a_p)$, letting it coincide with $p$. Now $p'$ is a path from $(1,1)$ to $(n+1-2s+2(n+4-2s),n+1) = (3n+9-6s,n+1)$. Conversely, a path $p'$ from $(1,1)$ to $(3n+9-6s,n+1)$ gives rise to a path from $(1,1)$ to $(n+1-2s,n+1)$ and the latter crosses the line $x=n+2-2s$ before the line $x=0$ if and only if $p'$ crosses the line $x=n+2-2s$ before possibly crossing the line $x=0$. Thus, we are counting the number of paths from $(1,1)$ to $(3n+9-6s,n+1)$ that do not cross the line $x=0$ before the line $x=n+2-2s$ is crossed, which is the same as the number of paths from $(0,0)$ to $(3n+8-6s,n)$ that do not cross the line $x=-1$ before the line $x=n+1-2s$ is crossed. Reflecting along the line $x=0$ we get that this number is the same as the number of paths from $(0,0)$ to $(6s-3n-8,n)$ that do not cross the line $x=1$ before the line $x=2s-n-1$ is crossed, which in turn is the same as the number of paths from $(8+3n-6s,0)$ to $(0,n)$ that do not cross the line $x=9+3n-6s$ before the line 
$x=2s-n-1+ 8+3n-6s = 2n-4s+7$ is crossed, that is, $e(n,8+3n-6s,2n-4s+7,9+3n-6s).$

For the paths that cross the line  $x=0$ before the line $x=n+2-2s$, we use a similar reflection argument to reduce to counting paths from 
$(1,1)$ to $2(s-4)-n+1,n+1)$ that do not cross the line $x=n+2-2s$ before crossing the line $x=0$, which is the same as the number of paths from $(n-2(s-4),0)$ to $(0,n)$ that do not cross the line $x=2n-4s+9$ before crossing the line $x=n-1-2(s-4)$, that is, $e(n,n-2(s-4),n-1-2(s-4),2n-4s+9)$.

By Lemma \ref{lemma:rec} we have
\begin{align*}
e(n,8+3n-6s,2n-4s+7,9+3n-6s) = \\
\sum_{k \geq 0} \left(  \binom {n}{3s-n-4-k(n-2s+4)} -   \binom {n}{3s-n-6 -k(n-2s+4)}\right) = \\
\sum_{k \geq 1} \left(  \binom {n}{s-k(n-2s+4)} -   \binom {n}{s-2 -k(n-2s+4)}\right)
\end{align*}
and
\begin{align*}
e(n,n-2(s-4),n-1-2(s-4),2n-4s+9) =\\ 
\sum_{k \geq 0} \left(\binom {n}{s-4-k(n-2s+4)} -  \binom {n}{3s-n-6-k(n-2s+4)}\right) = \\
\sum_{k \geq 0} \binom {n}{s-4-k(n-2s+4)} -  \sum_{k \geq 1} \binom {n}{s-2-k(n-2s+4)}.\\
\end{align*}
Writing $d=n-2s+4$, we get 
\begin{align*}
\binom{n}{s} - 2 \binom{n}{s-2} + u =\\
\sum_{k \geq 0} \left( \binom{n}{s-kd}  - 2\binom{n}{s-2-kd} + \binom{n}{s-4-kd} \right) = \\ 
\sum_{k = 0}^{\lceil s/3 \rceil } \left( \binom{n}{s-kd}  - 2\binom{n}{s-2-kd} + \binom{n}{s-4-kd} \right), 
\end{align*}
where the last equality uses the assumption $d \geq 3.$
\end{proof}

We are now ready to prove \cite[Conjecture 2]{CLN19}.

\begin{theorem} \label{thm:d2}

The Hilbert series of $R_{n,n+2,2}$ is equal to 
$1+c(n,1)t+c(n,2)t^2 + \cdots + c(n, \lceil \frac{n}{2} \rceil) t^{\lceil \frac{n}{2} \rceil}$, where $c(n, s)$ is the number of lattice paths inside the rectangle $(n+2-2s) \times (n+2)$ from the bottom left corner to the top right corner with moves of two types: $(x,y) \mapsto (x+1,y+1)$ or $(x-1,y+1)$. 

Moreover, for $s \leq \lceil \frac{n}{2} \rceil$, the number $c(n,s)$ is given by the explicit formula 
\[ \sum_{k = 0}^{\lceil s/3 \rceil }  \left( \binom{n}{s-kd}  - 2\binom{n}{s-2-kd} + \binom{n}{s-4-kd} \right),\]
where $d = n-2s+4.$
\end{theorem}

\begin{proof}
By Theorem \ref{thm:n+2equi}, we have  $D_{n,n+2,2} = \left\lceil \frac{n}{2} \right \rceil.$
Suppose first that the rectangle is degenerate, that is, $n+2-2s \leq 0 \Leftrightarrow \frac{n+2}{2} \leq s$, so that $D_{n,n+2,2} < s.$ Then $c(n,s) = 0$ is the dimension of the $s$'th graded piece of $R_{n,n+2,2}$, as claimed.

For the rest of the proof, we assume that 
$n+2-2s \geq 1.$ The explicit formula for $c(n,s)$ then follows from Lemma \ref{lemma:lattice}, so it is enough to show that the $s$'th graded component of the series of $R_{n,n+2,2}$ equals $c(n,s)$. By Theorem \ref{thm:n+2equi}, the dimension of the $s$'th graded component of $R_{n,n+2,2}$ is 
\begin{equation} \label{eq:p2}
\sum_{r=0}^{\lfloor (n-1)/2 \rfloor } \sum_{k=0}^{n-1-2r} (-1)^{k} 
\binom{n+2}{k}  \binom{- n r + 2 s r - 4 r - 2 k + s +n - 1}{n-1}.
\end{equation}
We begin by rewriting the upper limits in the summations.
Let $i$ be a positive integer and consider $k = n-1-2t + i$ in 
\[ 
\binom{n+2}{k}  \binom{- n r + 2 s r - 4 r - 2 k + s +n - 1}{n-1}.
\]

We have 
$- n r + 2 s r - 4 r - 2 k + n + s- 1 -(n-1)=-n r + 2r s - 2 n + s - 2 i + 2,$ and the degree of the Hilbert series of  $R_{n,n+2,2}$ is $\lceil n/2 \rceil$ by Theorem \ref{thm:n+2equi}, so from our assumption $s \leq (n+1)/2$ we get
$-n r + 2r s - 2 n + s - 2 i + 2 \leq -\frac{3}{2}n+r-2\,i+\frac{5}{2},$ which is negative when $r \leq \lfloor (n-1)/2 \rfloor.$ Moreover, using $n-2s+4 \geq 3$ we get that $-r(n-2s+4)-2k+s < 0$ when $r >  \lceil s/3 \rceil$, so we can write (\ref{eq:p2}) as 
\[ \sum_{r=0}^{\lceil s/3 \rceil } \sum_{k \geq 0} (-1)^{k} 
\binom{n+2}{k}  \binom{s-r(n-2s+ 4) - 2 k + n  - 1}{n-1}.\]
For fixed $r$, this is the coefficient of $t^{s-r(n-2s+4)}$ in \[\frac{(1-t^2)^{n+2}}{(1-t)^n} = (1-t^2)^2(1+t)^n,\]
and the coefficient of $t^{s-t(n-2s+4)}$ in the latter expression equals
\[ \binom{n}{s-r(n-2s+4)} - 2\binom{n}{s-r(n-2s+4)-2} + \binom{n}{s-r(n-2s+4)-4},\]
which concludes the proof.
\end{proof}
We remark that Cruz and Iarrobino \cite{DCI00} conjectured 
$c(2k,k) = 2^k$ and $c(2k+1,k+1) = 1$ which was proved by Sturmfels and Xu \cite{SX10} using the Verlinde formula, the latter of which we will return to in Section \ref{sec:mixed}. To recover this result from Theorem \ref{thm:d2}, we use the combinatorial interpretation. For $n = 2k$ we count the lattice paths inside the rectangle $2 \times  (2k + 2)$ starting from $(0,0)$ and ending in $(2,n+2)$. We have two choices at the points $(1,1), (1,3), \ldots, (1,2k-1)$, and the remaining moves are forced. This gives in total $2^k$ lattices paths. For $n=2k+1$ we have a $1 \times (2k+3)$ rectangle, and here every move is forced, giving exactly one lattice path.

We also remark that \cite[Conjecture 2]{CLN19} was settled in \cite{CLN19} for $s \leq 
g(n,2) = 
\left \lfloor \frac{n}{3} \right \rfloor+ 1$, and that  
Booth, Singh, and Vraciu \cite{BSV26} recently proved  the case $s = g(n,2)+1=\left \lfloor \frac{n}{3} \right \rfloor  + 2.$

\begin{corollary} \label{cor:exterior}
Let $f$ and $g$ be generic forms in the exterior algebra $E$ on an even number of generators $n$.
Then the Hilbert series of $E/(f,g)$ equals 
$1+c(n,1)t+c(n,2)t^2 + \cdots + c(n, \frac{n}{2} ) t^{ \frac{n}{2} },$
where $c(n,s)$ is the number of lattice paths inside the rectangle $(n+2-2s) \times (n+2)$ from the bottom left corner to the top right corner with moves of two types: $(x,y) \mapsto (x+1,y+1)$ or $(x-1,y+1)$, given by the explicit formula 
\[ \sum_{k = 0}^{\lceil s/3 \rceil }  \left( \binom{n}{s-kd}  - 2\binom{n}{s-2-kd} + \binom{n}{s-4-kd} \right),\]
where $d= n-2s+4$.
\end{corollary}
\begin{proof}
The proof of the statement follows by Theorem \ref{thm:d2} and \cite[Theorem 1]{CLN19}. 
\end{proof}

We remark that the interpretation of the Hilbert series for $R_{n,n+2,2}$ as a combinatorial object suggests that there may be combinatorial ways to describe the series of $R_{n,n+2,d}$ for $d \geq 3$ as well.

\section{The Cremona transformations}
A Cremona transform is a birational endomorphism of projective space, and used extensively for studying configurations of fat points, see for instance   \cite{LU06,D09}.

These techniques have been applied also for general linear forms by first using the Emsalem-Iarrobino correspondence,  then applying the Cremona transformations, and finally using the Emsalem-Iarrobino correspondence to go back to powers of general linear forms, see for instance 
\cite{HSS11, MM18, AA18, NT19, MN21}.
 However, this approach comes with redundant restrictions on the exponents of the general linear forms. For this reason, we state and prove the corresponding Cremona transformations directly for linear forms, although we still rely on the Emsalem-Iarrobino correspondence in the proofs.

\subsection{The first Cremona transform for powers of linear forms}

We begin with the following lemma.

\begin{lemma} \label{lemma:cremona1}
Let $\ell_1,\ldots,\ell_r$ be linear forms in $S$ such that $\ell_i=x_i$ for $i = 1,\ldots,n$, and such that $\ell_{i}$ has full support for $i=n+1,\ldots,r$. Let 
$d_1,\ldots,d_r$ be positive integers, let $d$ be a non-negative integer, and 
let $k=d_1+\cdots+d_n-n-2d.$ Suppose that $d_i+k \geq 0$ for $i=n+1,\ldots,r$ and let $F \in S_d.$ 

Then  \begin{equation} \label{eq:cremonaimplication}
\ell_i^{d_i} \circ F = 0 \text{ for } i=1,\ldots,r \,\,\, \Longrightarrow \,\,\, \overline{\ell_i}^{d'_i} \circ G = 0 \text{ for } i=1,\ldots,r, \end{equation}
where 
\[
G=
x_1^{d_1-1} \cdots x_n^{d_n-1} F(x_1^{-1},\ldots,x_n^{-1}),\]
\[
d'_i=\begin{cases}
d_i &\text{ if } 1 \leq i \leq n\\
d_i+k &\text{ if } n+1 \leq i \leq r,
\end{cases}
\]
and \[\overline{a_1 x_1+\cdots+a_n x_n} = a_1^{-1}x_1+ \cdots + a_n^{-1}x_n.\]
In particular, we have \[\dim_{\mathbf{k}} (S/(\ell_1^{d_1},\ldots, \ell_r^{d_r}))_d \leq \dim_{\mathbf{k}} (S/(\overline{\ell_1}^{d'_1},\ldots, \overline{\ell_r}^{d'_r}))_{d+k}.\]
\end{lemma}
\begin{proof}

To simplify the notation, let \[G= x_1^{d_1-1} \cdots x_n^{d_n-1} F(x_1^{-1},\ldots,x_n^{-1}).\]

Since $\ell_i^{d_i} \circ F=0$ it follows that the $x_i$-degree of $F$ is 
at most $d_i-1$ for $i=1,\ldots,n$.  Hence $G$ is a polynomial in $S_{d_1-1+\cdots+d_n-1-d} = S_{d+k}$,  whose $x_i$-degree is
at most $d_i-1$ for $i=1,\ldots,n$, so 
\[\overline{\ell_i}^{d_i} \circ G = 0 \text{ for } i = 1, \ldots, n.\]

We now proceed to the case $n+1 \leq i \leq r$. Fix one such $i$. Suppose first that $d-d_i+1 \leq 0.$ In this case $d_i > d$, so $d_i+k > d + k.$ Since $G \in S_{d+k}$, it follows that 
\[\overline{\ell_i}^{d_i+k} \circ G = 0.\]
Suppose instead that $d-d_i+1 > 0.$ 
By the Emsalem-Iarrobino correspondence, the inverse system of $(\ell_1^{d_1},\ldots,\ell_r^{d_r})$ of degree $d$ is 
\[(\mathfrak{p}_1^{d-d_1+1} \cap \cdots \cap \mathfrak{p}_r^{d-d_r+1})_d,\] (where $\mathfrak{p}_j^{d-d_j+1} = (1)$ if $d-d_j+1 \leq 0$).
It follows that $F$  can be written as a linear combination of elements of the form \[h=h_{1} \cdots h_{d-d_i+1} \cdot m,\] with $ h_j \in (\mathfrak{p}_i)_1 \text{ for }j = 1, \ldots, d-d_i+1$, and $m$ 
a monomial of degree $d_i-1$. If we can show that 
\[\overline{\ell_i}^{d_i+k} \circ x_1^{d_1-1} \cdots x_n^{d_n-1} h(x_1^{-1},\ldots,x_n^{-1}) =0,\] then we are done.

Since $h$ is also an element in $(\mathfrak{p}_1^{d-d_1+1} \cap \cdots \cap \mathfrak{p}_n^{d-d_n+1})_d,$ we can as before  assume that the $x_j$-degree is at most $d_j-1$ for $j = 1, \ldots, n,$ so $x_1^{d_1-1} \cdots x_n^{d_n-1} h(x_1^{-1},\ldots,x_n^{-1})$ is a polynomial in $S_{d+k}$, and that if it is the zero polynomial, then \ref{eq:cremonaimplication} holds.
Thus we can suppose that $x_1^{d_1-1} \cdots x_n^{d_n-1} h(x_1^{-1},\ldots,x_n^{-1})$ is non-zero of degree $d+k$.
Writing  $l_i = a_1 x_1 + \cdots + a_n x_n$  
we can assume that $h_j = a_{s_j} x_{t_j} - a_{t_j} x_{s_j}$ for $j=1,\ldots,d-d_i+1$, so that
\[h = m \cdot \prod_{j=1}^{d_i-1} (a_{s_j} x_{t_j} - a_{t_j} x_{s_j}).\]
We then have \begin{align*}
h(x_1^{-1},\ldots,x_n^{-1}) = m(x_1^{-1},\ldots,x_n^{-1}) \cdot \prod_{j=1}^{d_i-1} (a_{s_j} x_{t_j}^{-1} - a_{t_j} x_{s_j}^{-1}) = \\
m(x_1^{-1},\ldots,x_n^{-1}) \cdot \prod_{j=1}^{d_i-1} \frac{(a_{s_j} x_{s_j} - a_{t_j} x_{t_j}) }{x_{t_j} x_{s_j}} =\\
m(x_1^{-1},\ldots,x_n^{-1}) \cdot  \prod_{j=1}^{d_i-1} x_{t_j}^{-1} x_{s_j}^{-1} \cdot  \prod_{j=1}^{d_i-1} (a_{s_j} x_{s_j} - a_{t_j} x_{t_j}).
\end{align*}
Since $x_1^{d_1-1} \cdots x_n^{d_n-1} h(x_1^{-1},\ldots,x_n^{-1})$ is a polynomial, it follows that
\[x_1^{d_1-1} \cdots x_n^{d_n-1} \cdot m(x_1^{-1},\ldots,x_n^{-1}) \cdot  \prod_{j=1}^{d_i-1} x_{t_j}^{-1} x_{s_j}^{-1}\] is a monomial of degree $d_i-1+k$, which we call $m'.$
Thus \[x_1^{d_1-1} \cdots x_n^{d_n-1} h(x_1^{-1},\ldots,x_n^{-1}) = 
m' \prod_{j=1}^{d_i-1} (a_{s_j} x_{s_j} - a_{t_j} x_{t_j}).\]

Since $\overline{\ell_i} \circ (a_{s_j} x_{s_j} - a_{t_j} x_{t_j}) = 0$ for $j=1,\ldots,d-d_i+1$, it follows from the pigeonhole principle and the generalized Leibniz rule that \[\overline{\ell_i}^{d_i+k} \circ x_1^{d_1-1} \cdots x_n^{d_n-1} h(x_1^{-1},\ldots,x_n^{-1}) =0.\]

Finally, the inequality \[\dim_{\mathbf{k}} (S/(\ell_1^{d_1},\ldots, \ell_r^{d_r}))_d \leq \dim_{\mathbf{k}} (S/(\overline{\ell_1}^{d'_1},\ldots, \overline{\ell_r}^{d'_r}))_{d+k}\] follows since if $F_1,\ldots,F_m$ are linearly
independent, then so are 
\[ x_1^{d_1-1} \cdots x_n^{d_n-1} F_1(x_1^{-1},\ldots,x_n^{-1}), 
\ldots,
x_1^{d_1-1} \cdots x_n^{d_n-1} F_m(x_1^{-1},\ldots,x_n^{-1}).
\]

\end{proof}

\begin{lemma} \label{lemma:cremona1zero}
Let $\ell_1,\ldots,\ell_r$ be linear forms in $S$ such that $\ell_i=x_i$ for $i = 1,\ldots,n$, and such that $\ell_{i}$ has full support for $i=n+1,\ldots,r$. Let 
$d_1,\ldots,d_r$ be positive integers, let $d$ be a non-negative integer, and 
let $k=d_1+\cdots+d_n-n-2d.$ 

Suppose that $d+k<0$ or $d_j+k \leq 0$ for some $j \in \{n+1,\ldots,r\}$. Let $F \in S_d$ and suppose that 
$\ell_i^{d_i} \circ F = 0 \text{ for } i = 1,\ldots,r.$ Then $F = 0$, and
in particular, \[\dim_{\mathbf{k}} (S/(\ell_1^{d_1},\ldots, \ell_r^{d_r}))_d = 0.\]
\end{lemma}
\begin{proof}

Recall that a linear form supported on $x_1,\ldots,x_n$ is a strong Lefschetz element on $A=S/(x_1^{d_1},\ldots,x_n^{d_n})$, that the degree of the Hilbert series of the latter algebra is 
$d_1-1+\cdots+d_n-1$, and that it is symmetric along $(d_1+\cdots+d_n - n)/2$.

If $d+k < 0$, which is equivalent to $d_1-1+\cdots+d_n-1<d$, then $d$ is above the socle degree of $A.$ Thus $F=0$.

Suppose instead that $d+k \geq 0$ and that $d_j+k=0$. Since $\ell_j^{-k}$ is a strong Lefschetz element, the map induced by multiplication by $\ell_j^{-k}$ from degree $(d_1+\cdots+d_n - n)/2 -d_j/2=d+k$ to degree
$(d_1+\cdots+d_n - n)/2 + d_j/2 = d$ is a bijection. Hence, the inverse system in degree $d$ contains only the zero polynomial.
 
Thus, the proof is complete for the case $d_j+k =0$. The case $d_j+k<0$ follows from the previous case together with the inequality
$\dim_{\mathbf{k}} (S/(x_1^{d_1},\ldots,x_n^{d_n}, \ell^{-k}))_d \geq 
\dim_{\mathbf{k}} (S/(x_1^{d_1},\ldots,x_n^{d_n}, \ell^{d_j}))_d,$
which holds since  $-k>d_j$ by assumption.
\end{proof}

\begin{theorem} \label{thm:cremonafirst}
Let $\ell_1,\ldots,\ell_r$ be general linear forms in $S$,  let $d_1,\ldots,d_r$ be positive integers, let $d$ be a non-negative integer, and let $k =  d_1+\cdots+d_n-n-2d.$ 

If $d_j +k \leq 0$ for some 
$n+1 \leq j \leq r$ or $d-k<0$, then 
$\dim_{\mathbf{k}} (S/(\ell_1^{d_1},\ldots,\ell_r^{d_r}))_{d} = 0.$
Otherwise
\[
\dim_{\mathbf{k}} (S/(\ell_1^{d_1},\ldots,\ell_r^{d_r}))_{d} = 
\dim_{\mathbf{k}} S/(\ell_1^{d'_1},\ldots,\ell_r^{d'_r})_{d+k}, \]
where 
\[
d'_i=\begin{cases}
d_i &\text{ if } 1 \leq i \leq n\\
d_i+k &\text{ if } n+1 \leq i \leq r.
\end{cases}
\]
\end{theorem}

\begin{proof}
After a linear change of coordinates we can assume that $\ell_i = x_i$ for $i=1,\ldots,n$ and that
$\ell_i$ has full support for $i = n+1,\ldots,r$. The first statement now follows by Lemma \ref{lemma:cremona1zero}.

For the second statement, we get by Lemma \ref{lemma:cremona1} that  \[\dim_{\mathbf{k}} (S/(\ell_1^{d_1},\ldots, \ell_r^{d_r}))_d \leq \dim_{\mathbf{k}} (S/(\overline{\ell_1}^{d'_1},\ldots, \overline{\ell_r}^{d'_r}))_{d+k}.\] 
If we can show that 
\[
\dim_{\mathbf{k}} (S/(\overline{\ell_1}^{d'_1},\ldots, \overline{\ell_r}^{d'_r}))_{d+k} \leq
\dim_{\mathbf{k}} (S/(\ell_1^{d_1},\ldots, \ell_r^{d_r}))_d,\] 
then we are done. We do this by applying 
Lemma \ref{lemma:cremona1} with respect to $d'_1, \ldots, d'_{r}$ and $d' = d + k$. We need to check that
$d'_j + k' > 0$ for $j = n+1,\ldots,r$ and that $d'+k' \geq 0$, where $k' = d_1+\cdots+d_n-n-2d' =-k$.
We get $d'_j + k' = d_j + k - k = d_j$ and $d' + k' = d+k-k= d,$ so the assumptions in Lemma \ref{lemma:cremona1} are indeed satisfied. 
\end{proof}

\begin{remark}
It is possible to relate the inverse systems before and after the Cremona transformation. Indeed, with the notation as in Lemma \ref{lemma:cremona1} and Theorem \ref{thm:cremonafirst}, we can argue similarly as in the proof of Theorem  \ref{thm:cremonafirst} to see that   
if $d_j +k \leq 0$ for some 
$n+1 \leq j \leq r$ or $d-k<0$, then 
$(\ell_1^{d_1},\ldots,\ell_r^{d_r})^{-1} = \{0\}$, and
otherwise 
\[
F \in ( (\ell_1^{d_1},\ldots,\ell_r^{d_r})^{-1})_{d} 
\Leftrightarrow
x_1^{d_1-1} \cdots x_n^{d_n-1} F(x_1^{-1}, \ldots, x_n^{-1}) \in
((\overline{\ell_1}^{d'_1},\ldots, \overline{\ell_r}^{d'_r})^{-1})_{d+k}. \]
\end{remark}

It is useful to think of the first Cremona transformation as a reflection along the symmetry line of the underlying complete intersection, as illustrated in Figure \ref{fig:cremona}.

\begin{figure}[!ht]
\centering
\begin{tikzpicture}[x=0.37cm,y=0.02cm]
  \draw[->,gray] (-1,0) -- (33,0) node[right] {};
  \draw[->,gray] (0,0) -- (0,520) node[above] {};
  \foreach \x in {0,...,31}{
    \draw[gray!70] (\x,0) -- (\x,-5pt);
  }
  \foreach \x/\lab in {13/13,14/14,15/15,16/16,17/17,18/18,19/19}{
    \node[gray!30,below] at (\x,-5pt) {\lab};
  }
  \draw[dashed,thick,gray!70] (16,0) -- (16,520);
  \draw[thin,blue]
    plot coordinates {
      (0,1) (1,4) (2,10) (3,20) (4,35) (5,56) (6,84) (7,120)
      (8,165) (9,216) (10,270) (11,324) (12,375) (13,420)
      (14,456) (15,480) (16,489) (17,480) (18,456) (19,420)
      (20,375) (21,324) (22,270) (23,216) (24,165) (25,120)
      (26,84) (27,56) (28,35) (29,20) (30,10) (31,4) (32,1)
    };
  \foreach \x/\y in {
    0/1, 1/4, 2/10, 3/20, 4/35, 5/56, 6/84, 7/120,
    8/165, 9/216, 10/270, 11/324, 12/375, 13/420,
    14/456, 15/480, 16/489, 17/480, 18/456, 19/420,
    20/375, 21/324, 22/270, 23/216, 24/165, 25/120,
    26/84, 27/56, 28/35, 29/20, 30/10, 31/4, 32/1
  }{
    \filldraw[blue] (\x,\y) circle (0.7pt);
  }
  \draw[thin,red]
    plot coordinates {
      (0,1) (1,4) (2,10) (3,20) (4,35) (5,56) (6,84) (7,120)
      (8,165) (9,214) (10,262) (11,304) (12,335) (13,350)
      (14,344) (15,312) (16,249) (17,154) (18,60) (19,4)
    };
  \foreach \x/\y in {
    0/1, 1/4, 2/10, 3/20, 4/35, 5/56, 6/84, 7/120,
    8/165, 9/214, 10/262, 11/304, 12/335, 13/350,
    14/344, 15/312, 16/249, 17/154, 18/60, 19/4
  }{
    \filldraw[red] (\x,\y) circle (0.7pt);
  }
  \filldraw[red] (17,154) circle (2.0pt);
  \node[red,anchor=west] at (17.3,154) {$154$};
  \filldraw[red] (18,60) circle (2.0pt);
  \node[red,anchor=west] at (18.3,60) {$60$};
  \filldraw[red] (19,4) circle (2.0pt);
  \node[red,anchor=west] at (19.3,4) {$4$};
  \draw[thin,green!60!black]
    plot coordinates {
      (0,1) (1,4) (2,10) (3,20) (4,35) (5,56) (6,84) (7,118)
      (8,157) (9,196) (10,230) (11,254) (12,263) (13,252)
      (14,217) (15,154) (16,77) (17,16)
    };
  \foreach \x/\y in {
    0/1, 1/4, 2/10, 3/20, 4/35, 5/56, 6/84, 7/118,
    8/157, 9/196, 10/230, 11/254, 12/263, 13/252,
    14/217, 15/154, 16/77, 17/16
  }{
    \filldraw[green!60!black] (\x,\y) circle (0.7pt);
  }
  \filldraw[green!60!black] (15,154) circle (2.0pt);
  \node[green!60!black,anchor=east] at (14.7,154) {$154$};
  \draw[thin,magenta]
    plot coordinates {
      (0,1) (1,4) (2,10) (3,20) (4,35) (5,54) (6,76) (7,100)
      (8,125) (9,146) (10,159) (11,160) (12,145) (13,110)
      (14,60) (15,16) (16,1)
    };
  \foreach \x/\y in {
    0/1, 1/4, 2/10, 3/20, 4/35, 5/54, 6/76, 7/100,
    8/125, 9/146, 10/159, 11/160, 12/145, 13/110,
    14/60, 15/16, 16/1
  }{
    \filldraw[magenta] (\x,\y) circle (0.7pt);
  }
  \filldraw[magenta] (14,60) circle (2.0pt);
  \node[magenta,anchor=east] at (13.7,60) {$60$};
  \draw[thin,orange!80!black]
    plot coordinates {
      (0,1) (1,4) (2,10) (3,18) (4,27) (5,36) (6,45) (7,54)
      (8,63) (9,68) (10,65) (11,50) (12,27) (13,4)
    };
  \foreach \x/\y in {
    0/1, 1/4, 2/10, 3/18, 4/27, 5/36, 6/45, 7/54,
    8/63, 9/68, 10/65, 11/50, 12/27, 13/4
  }{
    \filldraw[orange!80!black] (\x,\y) circle (0.7pt);
  }
  \filldraw[orange!80!black] (13,4) circle (2.0pt);
  \node[orange!80!black,anchor=east] at (12.7,4) {$4$};
\end{tikzpicture}
\caption{
An illustration on how the first Cremona transform may be interpreted in terms of reflections along the symmetry line of the underlying complete intersection. The blue dots corresponds to the value of the Hilbert function for the complete
intersection $S/I$, where $S=k[x_1,x_2,x_3,x_4]$ and $I=(x_1^9,x_2^9,x_3^9,x_4^9)$, with a peak in degree $16$. The red dots corresponds to $S/(I+(\ell_5^9,\ell_6^9))$, the green dots corresponds to $S/(I+(\ell_5^{9-2},\ell_6^{9-2}))$, the magenta dots corresponds to $S/(I+(\ell_5^{9-4},\ell_6^{9-4}))$, and the orange dots corresponds to  $S/(I+(\ell_5^{9-6},\ell_6^{9-6}))$.
} \label{fig:cremona}
\end{figure}
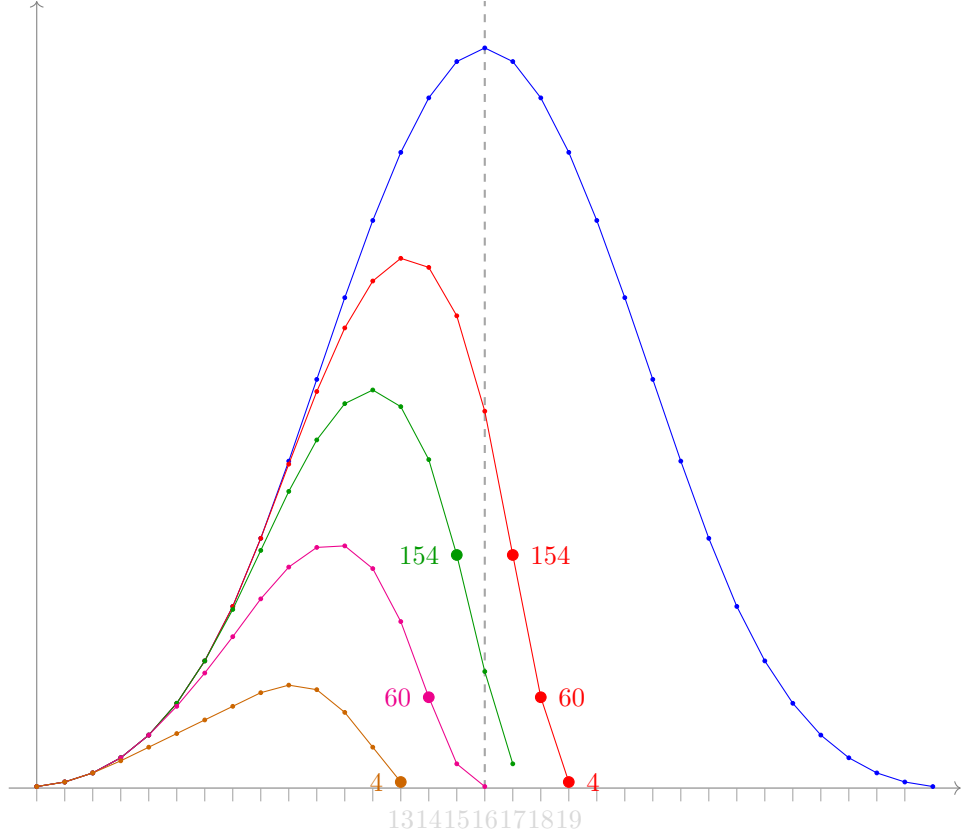

\subsection{The second Cremona transform for powers of linear forms}

\begin{lemma}  \label{lemma:cremona2}
Let $\ell_1,\ldots,\ell_r$ be linear forms in $S$ such that $\ell_i=x_i$ for $i = 1,\ldots,n$, and such that $\ell_{i}$ is supported by $x_n$ for $i=n+1,\ldots,r$. Let $d_1,\ldots,d_r$ be positive integers. Then we have

\[\ell_i^{d_i} \circ F = 0 \text { for } i =1, \ldots, r \,\,\,\Longleftrightarrow \,\,\,
\ell_i^{d'_i} \circ x_n F = 0 \text { for } i =1, \ldots, r,\]
where $d'_i = d_i$ for $i = 1,\ldots,n-1$ and $d'_i = d_i+1$ for $i=n,\ldots,r$.
\end{lemma}

\begin{proof}
The statement is clearly true for $i = 1,\ldots, n$. Fix one $i$ such that $n+1 \leq i \leq r.$

Since $\ell_i$ is supported by $x_n$, we can make a change of coordinates so that $\ell_i$ is the last variable and $x_n$ is a linear form supported by $\ell_i$. Suppose that $\ell_i^{d_i} \circ F = 0$. Then  the 
$\ell_i$-degree of $F$ is at most $d_i-1$, which implies that the $\ell_i$-degree of $x_n F$ is at most $d_i$, from which
it follows that $\ell_i^{d_i+1} \circ x_n F = 0.$ The other direction follows from a similar argument: Here we have 
$\ell_i^{d_i+1} \circ x_n F = 0$ implying that the $\ell_i$-degree of $x_n F$ is at most $d_i$, so the 
$\ell_i$-degree of $F$ is at most $d_i-1$, so $\ell_i^{d_i} \circ F = 0$.
\end{proof}

\begin{theorem}
Let $\ell_1,\ldots,\ell_r$ be linear forms in $S$ such that $\ell_i=x_i$ for $i = 1,\ldots,n$, and such that $\ell_{i}$ is supported by $x_n$ for $i=n+1,\ldots,r$. Let $d, d_1,\ldots,d_r$ be positive integers such that 
$d \geq d_1+\cdots+ d_{n-1} - (n-1).$

Then we have that 
\[ 
\{G \in S_{d+1}: \ell_i^{d'_i} \circ G = 0 \text { for } i =1, \ldots, r \} = \{ x_n F \in S_{d+1}:  \ell_i^{d_i} \circ F = 0 \text { for } i =1, \ldots, r \}, \]
where $d'_i = d_i$ for $i = 1,\ldots,n-1$ and $d'_i = d_i+1$ for $i=n,\ldots,r$. In particular, we have that
\[\dim_{\mathbf{k}} (S/(\ell_1^{d_1},\ldots,\ell_r^{d_r}))_d = \dim_{\mathbf{k}} (S/(\ell_1^{d'_1},\ldots,\ell_r^{d'_r}))_{d+1}.\]
\end{theorem}

\begin{proof}
Since $d \geq d_1+\cdots+ d_{n-1} - (n-1)$ we have that $d+1 > d_1+\cdots+ d_{n-1} - (n-1)$, so for a form $G$ of degree $d+1$ we have that
$\ell_i^{d_i} \circ G = 0$ for $i = 1,\ldots,n-1$ implies that $G = x_n F$. Thus
\[\{G \in S_{d+1} : \ell_i^{d'_i} \circ G = 0, i = 1,\ldots,r \} = 
\{x_n F \in S_{d+1} : \ell_i^{d'_i} \circ x_n F = 0, i = 1,\ldots,r \},\]
which equals 
\[\{x_n F \in S_{d+1} : \ell_i^{d_i} \circ F = 0, i = 1,\ldots,r \} \]
by Lemma \ref{lemma:cremona2}. 
\end{proof}

\section{The Hilbert function for powers of $n+2$ general linear forms for $n=3,4$} \label{sec:mixed}

We now give explicit values of the Hilbert function in the case of $3$ and $4$ generators by means of the Cremona transformations. We remark that these results could also be derived from Theorem \ref{thm:linearsystem}.

\subsection{The case $n=3$}

\begin{proposition} \label{prop:3uniform}
  The Hilbert function of
  $A = S/(\ell_1^d,\ell_2^d,\ell_3^d,\ell_4^d,\ell_5^d)$ for general
  linear forms $\ell_1,\ell_2,\ell_3,\ell_4,\ell_5$ is given by
  \[
    \dim_k A_i =
    \begin{cases}
      \binom{i+2}{2}-5\binom{i+2-d}{2}, & i\le (5d-4)/3\\
      5\binom{2d-i-1}{2}+1, & (5d-6)/3< i
      \le 2d-2\\
      0 & i>2d-2
    \end{cases}
  \]
\end{proposition}

\begin{proof}
  The Hilbert function in the statement agrees with the Hilbert
  function of an ideal generated by five general forms of degree $d$
  when $i\le (5d-4)/3$, which by Fr\"oberg's Conjecture has Hilbert
  series
  \[
    \left[\frac{(1-t^d)^5}{(1-t)^3}\right]
  \]
  When we are $j$ steps to the left of the socle degree $2(d-1)$, we
  are in degree $2d-j-2$ and this satisfies the inequality when
  $d\le 3j+2$.

  We can here use the Cremona transform to get to the inverse system
  in degree $d+j-1$ of
  $(\ell_1^{1+2j},\ell_2^{1+2j},\ell_3^d,\ell_4^d,\ell_5^d)$ with
  $k = 3(d-1)-2(2d-2-j) = 1+2j-d$. We can apply it again to get to the
  inverse system in degree $3j$ of
  $(\ell_1^{2j+1},\ell_2^{2j+1},\ell_3^{2j+1},\ell_4^{2j+1},\ell_5^d)$.
  Now we can deduce that this has the expected dimension if $d\ge 3j$
  since the almost complete intersection given by the first four has
  the expected Hilbert function and the last form give independent
  conditions in degree $d$. By the formula in Fr\"oberg's Conjecture,
  we get that this value is
  \[
    \binom{3j+2}{2}-4\binom{j+1}{2} = 5\binom{j+1}2+1
  \]
  when $d>3j$ and one less when $d = 3j$. Since $i = 2(d-1)-j$, we
  have $j=2d-2-i$ and we get
  \[
    \binom{3j+2}{2}-4\binom{j+1}{2} = 5\binom{2d-i-1}2+1.
  \]
  For $i \le 2(d-1)-\lfloor d/3\rfloor = \lfloor (5d-4)/3\rfloor$, we
  have that the Hilbert function agrees with the one given by general
  forms.

  Since this is in degree lower than the degrees of the Koszul
  syzygies and the expected Hilbert function is positive, this means
  that there are no syzygies among general linear forms and we have
  proven that there are no syzygies in this degree between
  $\ell_1^d,\ell_2^d,\ell_3^d,\ell_4^d,\ell_5^d$. Hence there can be
  now syzygies in lower degrees between them and the Hilbert function
  has to agree with the Hilbert function of five general forms of
  degree $d$.
\end{proof}

In three variables, we can also get the general case using only the Cremona transformations.  The general picture in three variables and five forms is that the only obstruction we get from getting the expected value of the Hilbert function is the injectivity of the multiplication by the quadric passing through the five points on the multigraded ring obtained from the inverse system, i.e., the Cox-Nagata ring described by Sturmfels and Xu.

This is what we saw in Proposition~\ref{prop:3uniform} and we will now establish this in the general case.  We can assume that $d_1\le d_2\le d_3\le d_4\le d_5$. The general strategy is to apply two Cremona transformations with $k = d_1+d_2+d_3-3-2d$ and $k' = d_1+(d_4+k)+(d_5+k) -3 -2(d+k) = d_1+d_4+d_5-3-2d$. This will lead to the inverse system of
\[
\left(\ell_1^{d_1},\ell_2^{d_2+k'},\ell_3^{d_3+k'},\ell_4^{d_4+k},\ell_5^{d_5+k}\right)
\]
in degree $d+k+k'$. When $d_1\ge d+k+k'$, this will have the expected
Hilbert function given by the Fr\"oberg Conjecture. We have that
\[
  d_1\ge d+k+k'\quad\Longleftrightarrow\quad \sum_{i=1}^5 d_i \le 6 +
  3d
\]
We introduce the notation 
$f_d(d_1,d_2,d_3,d_4,d_5)$
for the expected value of the Hilbert function according to
Fr\"oberg's Conjecture, if it is non-negative, and recall from Section \ref{sec:prel} that
\[
  f_d(d_1,d_2,d_3,d_4,d_5) = \sum_{k=0}^5 (-1)^k \sum_{1\le
    i_1<i_2<\dots<i_k} \binom{d+2-d_{i_1}-\cdots-d_{i_k}}{2}.
\]

\begin{proposition} 
  The Hilbert function of 
  $A =  S/(\ell_1^{d_1},\ell_2^{d_2},\ell_3^{d_3},\ell_4^{d_4},\ell_5^{d_5})$
  for general linear forms $\ell_1,\ell_2,\ell_3,\ell_4,\ell_5$ is given by
  \[
    \dim_k A_d =
    \begin{cases}
      f_d(d_1,d_2,d_3,d_4,d_5)
      , & 3d+4\le s \\
      f_{d-2j}(d_1-j,d_2-j,d_3-j,d_4-j,d_5-j),
       &      3d+6-d_1 \le  s< 3d+4 \\
      0 & s<3d+6-d_1
    \end{cases}
  \]
  where $s = d_1+d_2+d_3+d_4+d_5, j = 3d+5-s,$ and
  $d_1= \min\{d_1,d_2,d_3,d_4,d_5\}$.  
\end{proposition}

\begin{proof}
  When we apply the two Cremona transformations and $s<3d+6-d_1$, we
  get to the inverse system of
  $(\ell_1^{d_1},\ell_2^{d_2+k'},\ell_3^{d_3+k'},\ell_4^{d_4+k},\ell_5^{d_5+k})$
  in degree $d+k'+k = d_1+s-3d-6<0$. Hence we have that the value of the Hilbert function is zero.

  When $3d+6-d_1\le s <3d+4$, we can compare the inverse system of  
  \[(\ell_1^{d_1},\ell_2^{d_2},\ell_3^{d_3},\ell_4^{d_4},\ell_5^{d_5})\]
  in degree $d$ with the inverse system of the ideal
  \[(\ell_1^{d_1-j},\ell_2^{d_2-j},\ell_3^{d_3-j},\ell_4^{d_4-j},\ell_5^{d_5-j})\]
  in degree $d-2j$. The Cremona transformations of these two will go to the inverse system of $(\ell_1^{d_1},\ell_2^{d_2+k'},\ell_3^{d_3+k'},\ell_4^{d_4+k},\ell_5^{d_5+k})$ and
  $(\ell_1^{d_1-j},\ell_2^{d_2+k'},\ell_3^{d_3+k'},\ell_4^{d_4+k},\ell_5^{d_5+k})$
  in degree $d+k+k'$. When $d_1-j>d+k+k'$, these are equal. We have that
  \[
    d_1-j>d+k+k' \quad \Longleftrightarrow \quad 3d+5-j\ge
    \sum_{i=1}^5 d_i.
  \]
  Observe that in the range where $36+6-d_1\le s< 3d+4$, we have $2\le j\le d_1-1$.

  It remains to consider the range $3d+4\le s$. It suffices to prove the statement for the three cases $s=3d+4$, $s=3d+5$ and $s=3d+6$, since if we have equality with the Hilbert function of generic forms in this degree, we will also have it in lower degrees, since this means that we have only Koszul syzygies in degree $d$.

  First, we consider the case when $d_1+d_2>d$. This means that there are no Koszul relations and
  \begin{equation}\label{eq:FC}
    f_d(d_1,d_2,d_3,d_4,d_5) =\binom{d+2}{2}-\sum_{i=1}^5
    \binom{d+2-d_i}{2}.
  \end{equation}
  The Cremona transformation relates this to
  \begin{multline*}
    f_{d+k+k'}(d_1,d_2+k',d_3+k',d_4+k,d_5+k) \\ =
    f_{d_1+s-3d-6}(d_1,s-3-2d-d_3,s-3-2d-d_2,s-3-2d-d_5,s-3-2d-d_4)
  \end{multline*}
  We let $j = 3d+6-s$, for $j = 0,1,2$ and since there are no Koszul
  relatation here either, we get
  \begin{multline}
    \label{eq:h2}
    \binom{d_1-j+2}{2} - \binom{d_1-j+2-d_1}{2} -
    \sum_{i=2}^5\binom{d_1-j+2-(s-3-2d-d_i)}{2} \\= \binom{d_1-j+2}{2}
    - \binom{2-j}{2} - \sum_{i=2}^5\binom{d_1+d_i-d-1}{2}
  \end{multline}
  In this range, we can expand the binomial coefficients as polynomials and verify that the two expressions are equal.

  When $d_1+d_i\le d$, for some $i=2,3,4,5$, we get a Koszul relatation and we have to add the term
  \[
    \binom{d+2-d_1-d_i}{2}.
  \]
  in the right hand side of (\ref{eq:FC}). Since $d_1+d_i\le d$, we also get that the term
  \[
    \binom{d_1+d_i-d-1}{2}
  \]
  in (\ref{eq:h2}) is zero. These two changes cancel out and we still have equality.

  If we have a second Koszul relation not involving $d_1$, we will add
  \[
    \binom{d+2-d_i-d_j}{2}
  \]
  in (\ref{eq:FC}), but then we will add the same term to the right hand side of (\ref{eq:h2}). For example
  \[\begin{split}
      d+k+k'+2-(d_4+k)-(d_5+k) \\=
      d+2+(d_1+d_4+d_5-3-2d)-(d_1+d_2+d_3-3-2d) -d_4-d_5\\ =
      d+2-d_2-d_3
    \end{split}
  \]
  and
  \[
    d+k+k'+2-(d_2+k')-(d_4+k) = d+2+d_2-d_4.
  \]
  If there was a third order Koszul syzygy, we would have $d\ge d_1+d_2+d_3$ and the Hilbert function of
  $S/(\ell_1^{d_1},\ell_2^{d_2},\ell_3^{d_3})$ would be zero in degree $d$. We have that
  \[\begin{split}
      d_1+d_2+d_3\le d \quad\Longleftrightarrow\quad
      d_1+(d_4+k)+(d_5+k) \le d+d_1+d_4+d_5-3-2d +k \\
      \quad\Longleftrightarrow\quad d_1+(d_4+k)+(d_5+k) \le d+k+k'
    \end{split}
  \]
  which shows that $f_d(d_1,d_2,d_3,d_4,d_5)=f_{d+k+k'}(d_1,d_2+k',d_3+k',d_4+k,d_5+k)=0$ in this case.
\end{proof}

\subsection{The case $n=4$}
In the case of four variables, we need just a little bit more than just the Cremona transformations and this is provided by a result of Sturmfels and Xu. 

\begin{proposition}\label{proposition:4uniform}
  The Hilbert function of 
  $R_{4,6,d} = S/(\ell_1^d,\ell_2^d,\ell_3^d,\ell_4^d,\ell_5^d,\ell_6^d)$ for
  general linear forms $\ell_1,\ell_2,\ell_3,\ell_4,\ell_5,\ell_6$ is
  given by
  \[
    \dim_k [R_{4,6,4}]_i =
    \begin{cases}
      \binom{i+3}{3}-6\binom{i+3-d}{3}, & i\le 2(d-1)\\
      \binom{12d-5i-9}{3}-6\binom{7d-3i-5}{3}
      , & 10(d-1)< 5i\le 12(d-1)\\
      0 & 5i>12(d-1)
    \end{cases}
  \]

\end{proposition}

\begin{proof}
  In degree $i = 2(d-1)$, we have that the dimension is given by the
  Verlinde formula (cf. Theorem~\ref{thm:Verlinde} and
  Lemma~\ref{lemma:Verlinde} below) which agrees with the formula
  given by Fr\"oberg's Conjecture, i.e.,
  \[
    \binom{2(d-1)+3}{3}-6\binom{2(d-1)+3-d}{3} =
    \binom{2d+1}{3}-6\binom{d+1}{3} = \frac{d(d^2+2)}{3}.
  \]
  Since degree $i=2(d-1)$ is below the degree of the Koszul syzygies, this means that there are no syzygies between $d$th powers of six general linear forms in degrees less than or equal to $2(d-1)$.

  For degree $i=2(d-1)+j$, where $j>0$, we can use the Cremona transform to get to the inverse system in degree $2d-2-j$ of
  $(\ell_1^{d-2j},\ell_2^{d-2j},\ell_3^d ,\ell_4^d ,\ell_5^d
  ,\ell_6^d)$ with $k = 4(d-1)-2(2d-2+j) = -2j$. We can continue with the Cremona transform to the inverse system in degree $2d-2-3j$ of
  $(\ell_1^{d-2j},\ell_2^{d-2j},\ell_3^{d-2j} ,\ell_4^{d-2j} ,\ell_5^d
  ,\ell_6^d)$ with $k = 2(d-2j)+2d-4-2(2d-2-j) = -2j$. We can get back to uniform powers by one more application of the Cremona transform and we get to the inverse system in degree $2d-2-5j$ of
  $(\ell_1^{d-2j},\ell_2^{d-2j},\ell_3^{d-2j} ,\ell_4^{d-2j}
  ,\ell_5^{d-2j} ,\ell_6^{d-2j})$. When $2d-5-j<0$, we conclude that this is zero-dimensional. Otherwise, we have landed in a degree where we have proven above that the dimension equals the Hilbert function given by six general forms of degree $d-2j$.
\end{proof}

The following theorem is due to Sturmfels and Xu \cite[Thm
7.2]{SX10}. 
\begin{theorem}\label{thm:Verlinde}
  For $n = d+2$ and $G$ generic we have
  \[
    \psi (dl,2l,\dots,2l) = \frac{1}{2l+1}\sum_{j=0}^{2l}
    (-1)^{dj}\left(\sin\left(\frac{2j+1}{4l+2}\pi\right)\right)^{-d}.
  \]
  Here $l$ can be a half integer if $d$ is even but must be an integer if $d$ is odd.
\end{theorem}
  
\begin{lemma}\label{lemma:Verlinde}
 The dimension of $R_{4,6,d}$ in degree $2(d-1)$ equals
   \[
    \frac{d(d^2+2)}{3}.
  \]
\end{lemma}

\begin{proof}
  We need to apply Theorem \ref{thm:Verlinde} in the case when $n=6$, $d=4$ and
  $2l=d-1$.
  \[
    \psi (4(d-1),d-1,\dots,d-1) = \frac{1}{d}\sum_{j=0}^{d-1}
    (-1)^{4j}\left(\sin\left(\frac{2j+1}{2d}\pi\right)\right)^{-4}.
  \]
  We will now prove that
  \[
    \frac{1}{d} \sum_{j=0}^{d-1}
    \left(\sin\left(\frac{2j+1}{2d}\pi\right)\right)^{-4} =
    \frac{d^3+2d}{3}.
  \]
  For odd $d$, we can look at $A = \mathbb C[t]/(t^d-1)$. The multiplication by $q = (t+t^{-1})/2$ on $A$ is a linear operator $L$ which has eigenvalues $\cos(2\pi j/d)$ for $j=0,1,\dots,d-1$. Since $d$ is odd, we have that the absolute values of these eigenvalues equals $\sin(((2j+1)\pi/2d)$, $j = 0,1,\dots,d-1$. Hence the sum we need to compute is the trace of $L^{-4}$ divided by $d$. This equals the degree $0$ part of $q^{-4}$. We can write $q^{-4}$ as
  \[
    \left(\frac{2}{t+t^{-1}}\right)^4=
    \left(\frac{t^d+t^{-d}}{t+t^{-1}} \right)^4=
    \left(t^{d-1}-t^{d-3}+\cdots -t^{3-d}+t^{1-d}\right)^4.
  \]
  The constant term in this is given by the Hilbert function in degree $2(d-1)$ of the complete intersection of type $(d,d,d,d)$ minus the Hilbert function of the same complete intersection in degrees $d-2$ and $3d-2$. This is exactly the value given by the Fr\"oberg Conjecture as seen above.

  For even $d$, we'll use inducion on $d$ but we also need the corresponding trace of $L^{-2}$. This is given by $d$ as it is the Hilbert function in degree $d-1$ of the complete intersection of
  type $(d,d)$.

  We can use the trigonometric identity
  \[
    \frac{1}{\sin^4(x)} + \frac{1}{\cos^4(x)} =
    \frac{1-2\sin^2(x)\cos^2(x)}{\sin^4(x)\cos^4(x)} =
    \frac{16}{\sin^4(2x)} - \frac{8}{\sin^2(2x)}
  \]
  to see that
  \[\begin{split}
      \sum_{j=0}^{d-1}
      \left(\sin\left(\frac{2j+1}{2d}\pi\right)\right)^{-4} =
      \sum_{j=0}^{d/2-1}
      \left(\sin\left(\frac{2j+1}{2d}\pi\right)\right)^{-4} +
      \sum_{j=0}^{d/2-1}\left(\cos\left(\frac{2j+1}{2d}\pi\right)\right)^{-4}
      \\= 16\sum_{j=0}^{d/2-1}
      \left(\sin\left(\frac{2j+1}{d}\pi\right)\right)^{-4} -
      8\sum_{j=0}^{d/2-1}
      \left(\sin\left(\frac{2j+1}{d}\pi\right)\right)^{-2}
    \end{split}
  \]
  By induction on $d$ we can assume that
  \[
    \sum_{j=0}^{d/2-1}
    \left(\sin\left(\frac{2j+1}{d}\pi\right)\right)^{-4} =
    \frac{(d/2)^4+2(d/2)^2}{3}\]
     and \[
    \sum_{j=0}^{d/2-1}
    \left(\sin\left(\frac{2j+1}{d}\pi\right)\right)^{-2} =
    \left(\frac{d}{2}\right)^2
  \]
  which gives
  \[
    \sum_{j=0}^{d-1}
    \left(\sin\left(\frac{2j+1}{2d}\pi\right)\right)^{-4} = 16
    \frac{(d/2)^4+2(d/2)^2}{3} - 8 \left(\frac{d}{2}\right)^2 =
    \frac{d^4+2d^2}{2}.
  \]
\end{proof}

\section{Results for $m\geq n+3$ $d$'th powers of general linear forms}

Cremona transformations allow us to draw some conclusions also for  $d$'th powers of general linear forms for some $n$ and some $m > n+2$. 

In the following propositions, the idea is to use these transformations to first give an upper bound $d_0=d_0(n,d)$ on the degree of the Hilbert series of $R_{n,m,d}$, and then show that the dimension in degree $d_0$ is positive. 

For the cases $(n,m) \in \{(3,6), (3,7), (3,8), (4,7)\}$, we determine the dimension in degree $d_0$ completely. We show that the dimension is a binomial coefficient for $d$ large enough, while for lower values, the dimension is given by a coefficient in a series predicted by Fr\"oberg. For $(n,m) \in \{(3,6), (3,7)\}$, where there are only a few number of values of $d$ for which the coefficient is not a binomial coefficient, we provide the coefficients explicitly. 

For $(n,m) = (5,8)$ the dimension is again a binomial coefficient for $d$ large enough, and for the remaining values of $d$, some coefficients are given via a series predicted by Fr\"oberg. For the remaining coefficients we instead provide intervals for them. For $(n,m) = (6,9)$, we give a lower bound on the dimension in degree $d_0$. We also give a lower bound on $d_0(k^2-3,d)$ for $k \geq 4$. 

 To simplify the proofs, we introduce the notation 
\[ [n,d_1^{i_1} \cdots d_m^{i_m};D] \]
for the inverse system in degree $D$ of the ideal 
\[ \left(\ell_{1,1}^{d_1}, \ldots, \ell_{1,i_1}^{d_1}, \ldots, \ell_{m,1}^{d_m}, \ldots, \ell_{m,i_m}^{d_m} \right),\]
where each $\ell_{i,j}$ is a general linear form in $S$.

\subsection{The case $n=3$}

\begin{proposition} \label{prop:36}
    For $d\ge 1$, write $12(d-1) = 7d_0+d_1$, where $0\le d_1<7$. Then $d_0$ is the degree of the Hilbert series of $R_{3,6,d}$ and the leading coefficient is
    \[
    c = \binom{d_1+2}{2}
    \]
    except for $d\in \{2,3,5,6,12\}$ where it is given by the following table.
    \begin{center}
        \begin{tabular}{|c|r|r|r|r|r|r|r|} \hline
            $d$ &$2$&$3$&$5$&$6$&$12$\\ \hline
           $c$ & $1$&$4$&$10$&$9$&$22$\\        \hline
           \end{tabular}
    \end{center}
\end{proposition}

\begin{proof}
Let $D$ be a non-negative integer. We start with the inverse system 
\[[3,d^6;D] \text{\quad with \quad} k = 3(d-1) - 2D.\]
A Cremona transformation gives 
\[[3,(d+k)^3 d^3; D+k]
\text{ \quad with \quad} k' = 3(d-1)-2D+3k-2k=2k.\]
A second Cremona transformation gives 
\[[3,(d+2k)^3(d+k)^3; D+3k] \text{\quad with \quad } k'' = 3(d-1)-2D+6k-6k=k,\]
and a third and final Cremona transformation gives 
\[[3,(d+2k)^6; D+4k].\]
If $D+4k<0$ the inverse system is the zero set. This happens when 
\[d+4 \cdot (3(d-1)-2D)  < 0 \Longleftrightarrow 
\left \lfloor \frac{12(d-1)}{7} \right \rfloor < D
\Longleftrightarrow
d_0 < D.
\]
Thus, the degree of the series of $R_{3,6,d}$ is at most $d_0$. Since 
\[d_0 + 4k =  12(d-1) - 7d_0 = d_1,\]
the dimension of $R_{3,6,d}$ in degree $d_0$ equals the dimension of
$R_{3,6,d+2k}$ in degree $d_1$. If $d+2k > d_1$, this is the binomial coefficient $\binom{2+d_1}{2}$, and  this happens when  
\[7(d-1) + 1 - 4d_0 > d_1 \Longleftrightarrow d_0 > 5d_1 -12,\]
which one can check holds for all $d \geq 13$, and for $d < 12$ except for $d \in \{2,3,5,6,12\}$. A computation with the ideal generated by the $d$'th powers of $x_1,x_2,x_3, x_1+x_2+x_3, 2x_1+3x_2+5x_3, 7x_1+11x_2+13x_3$ gives the same series as the series predicted by Fr\"oberg, and the leading coefficient is the claimed one. This finishes the proof.
\end{proof}

Notice that it would have been enough to check the coefficient of $d_1$ for $d+2k$'th powers of the specialization in the proof of Proposition \ref{prop:36}, for the exceptional values of $d$. This method will be used in the next proposition, where we have more special cases to deal with, in order to save some calculations.

\begin{proposition} \label{prop:37}
For $d \geq 1$, write $21(d-1) = 13d_0 + d_1$, where $0\le d_1<13$. Then $d_0$ is the degree of the Hilbert series of $R_{3,7,d}$ and the leading coefficient is
\[c=\binom{d_1+2}{2},\] except for the values of $d$ given by the following table. 
\[
\begin{tabular}{|c|ccccccccccccccccc|}
\hline
$d$ & $2$ & $3$ & $4$ & $5$ & $7$ & $8$ & $12$ & $13$ & $15$ & $20$ & $23$ & $28$ & $30$ & $33$ & $38$ & $43$ & $48$\\
\hline
$c$ & $3$ & $3$ & $8$ & $7$ & $13$ & $8$ & $24$ & $14$ & $24$ & $34$ & $29$ & $38$ & $57$ & $48$ & $59$ & $71$ & $84$\\
\hline
\end{tabular}
\]
\end{proposition}

\begin{proof}
We first perform the following sequence of Cremona transformations.
\[
\begin{aligned}
&[3,d^7;D] && k = 3(d-1)-2D  \\
&[3,(d+k)^4 d^3;D+k] && k^{(1)} = k +3k-2k=2k \\
&[3,(d+3k)^1(d+2k)^3(d+k)^3;D+3k] && k^{(2)} = k + 7k-6k=2k \\
&[3,(d+4k)(d+3k)^4(d+2k)^2;D+5k] && k^{(3)} = k + 10k-10k=k \\
&[3,(d+4k)^3(d+3k)^4; D+6k]&& k^{(4)} = k + 12k-12k=k \\
&[3,(d+4k)^7;D+7k].
\end{aligned}
\]
We have $D+7k < 0$ if and only if 
\[
D+7(3(d-1)-2D) < 0 
\Longleftrightarrow 
 \left \lfloor \frac{21(d-1)}{13} \right\rfloor < D
 \Longleftrightarrow 
 d_0 < D.
\]
Thus the degree of the series of $R_{3,7,d}$ is at most $d_0$. We need to look at degree $D=d_0$ and we get $k = (13d_0+d_1)/7-2d_0 = (d_1-d_0)/7$, and $d_0 + 7k =  d_1$. Hence 
the dimension of $R_{3,7,d}$ in degree $d_0$ equals the dimension of
$R_{3,7,d+4k}$ in degree $d_1$. If $d+4k > d_1$, this is the binomial coefficient $\binom{2+d_1}{2}$, and  this happens when  
\[13(d-1) + 1 - 8d_0 > d_1 \Longleftrightarrow d_0 > 8d_1 -21,\]
which one can check holds for all $d \geq 49$, and for $d < 49$ except for 
\[ d \in S=\{2,3,4,5,7,8,12,13,15,20,23,28,30,33,38,43,48 \}. \]

When $d < 49$, it is straightforward to check  that $d+4k \leq 12$. To deal with the exceptional cases, we make, for each such value of $d+4k$ a calculation with 
the ideal generated by the $d+4k$'th powers of
$ x_1,x_2,x_3,x_1+x_2+x_3, 2x_1+3x_2+5x_3,7x_1+11x_2+13x_3,17x_1+19x_2+29x_3,$ 
which gives the same series as predicted by Fr\"oberg for general forms of these degrees, with the coefficient in degree $d_1$, for each $d \in S$, being the one claimed in the proposition. 
\end{proof}

In the next proposition we will have a large number of special cases, and therefore, we will present the result in a more compact way.

\begin{proposition} \label{prop:38}
For $d \geq 1$, write $48(d-1) = 31d_0 + d_1$, where $0\le d_1<31$. Then $d_0$ is the degree of the Hilbert series of $R_{3,8,d}$ and the leading coefficient is the coefficient of
$t^{d_1}$ in 
\[\left[\frac{(1+t^{1+31(d-1) - 20d_0})^8}{(1-t)^3} \right],\]
which equals the binomial coefficient
\[\binom{d_1+2}{2}\] 
 when $d_0\ge 17d_1$, which in particular holds for all $d >300$. \end{proposition}

\begin{proof}
Cremona transformations gives 
\[
\begin{aligned}
&[3,d^8;D] && k = 3(d-1)-2D  \\
&[3,(d+k)^5 d^3;D+k] && k^{(1)} = k +3k-2k=2k \\
&[3,(d+3k)^2(d+2k)^3(d+k)^3;D+3k] && k^{(2)} = k + 8k-6k=3k \\
&[3,(d+5k)^2(d+4k)^3(d+3k)^2(d+2k);D+6k] && k^{(3)} = k + 14k-12k=3k \\
&[3,(d+7k)^2(d+6k)^2(d+5k)^3(d+4k); D+9k]&& k^{(4)} = k + 20k-18k=3k \\
&[3,(d+9k)^2(d+8k)^3(d+7k)^3(d+6k); D+12k]&& k^{(5)} = k + 25k-24k=2k \\
&[3,(d+10k)^2(d+9k)^4(d+8k)^3; D+14k]&& k^{(6)} = k + 28k-28k=k \\
&[3,(d+10k)^3(d+9k)^5; D+15k]&& k^{(7)} = k + 30k-30k=k \\
&[3,(d+10k)^8; D+16k].&&
\end{aligned}
\]

We have $D+16k < 0$ if and only if 
\[
D+16(3(d-1)-2D) < 0 
\Longleftrightarrow 
 \left \lfloor \frac{48(d-1)}{31} \right\rfloor < D
 \Longleftrightarrow 
 d_0 < D.
\]
Thus the degree of the series of $R_{3,8,d}$ is at most $d_0$ and we look in degree $D = d_0$, where we get $k = (31d_0+d_1)/16-2d_0 = (d_1-d_0)/16$ and $d_0+16 k = d_1$. Hence 
the dimension of $R_{3,8,d}$ in degree $d_0$ equals the dimension of
$R_{3,8,d+10k}$ in degree $d_1$. We have that $d+10k = 1+ (31d_0+d_1)/48 + 10(d_1-d_0)/16 = 1+ (d_0+31d_1)/48$, so 
\[
d+10k>d_1 \Longleftrightarrow \frac{d_0+31d_1}{48}\ge d_1 \Longleftrightarrow d_0\ge 17 d_1.
\]
Hence the leading coefficient equals 
\[
\binom{d_1+2}{2}
\]
when $d_0\ge 17d_1$. This inequality holds for all $d\ge 301$. For $d < 301$, we first observe that $d+10k=1+(d_0+31d_1)/48 < 1+ d_1 \leq 31$ since $d_0<17d_1$. We use the ideal generated by the $d+10k$'th powers of 
$x_1,x_2,x_3,x_1+x_2+x_3, 2x_1+3x_2+5x_3,7x_1+11x_2+13x_3,17x_1+19x_2+29x_3, 31x_1+37x_2+43x_3,$
to check by Macaulay2 computations that it gives the same series as predicted by Fr\"oberg for general forms of these degrees, and that the coefficient of $t^{d_1}$ is non-zero for the corresponding value of $d_1$.

\end{proof}
One can check that the series of $R_{3,8,d}$ has the same series as predicted by Fr\"oberg for $d \leq 62$, differs in degree $d=63$, and that the two series occasionally agree up to degree $d=362$. From this spot the series differs for $d \in \{363, \ldots,500\}$.

\begin{proposition} \label{prop:39}
The Hilbert series of $R_{3,9,d}$ is equal to the series predicted by Fr\"oberg for $9$ generic forms of degree $d$. In particular, when $d=2k+1$, the degree of the Hilbert series of $R_{3,9,d}$ is equal to $3k$, with leading coefficient equal to $1$, and when $d=2k$, the degree is $3k-2$, with leading coefficient equal to $3k$.

\end{proposition}

\begin{proof}

For the first part of the statement, we refer to Iarrobino \cite{I97} pp. 323, who cites \cite{H94} and \cite{H89} for the corresponding statements for fat point ideals. We also remark that since $9$ is a square, this is also a case covered by Ciliberto and Miranda \cite{CM06}.

For the second part, we first consider $d=2k+1$ where we have 
\[\binom{2+3k}{2} - 9\binom{k+1}{2} = 1 \text{ and }\binom{2+3k+1}{2} - 9\binom{k+2}{2} = -6k-6<0.\]
In the case $d=2k$, we have
\[
\binom{3k-2+2}{2}-9\binom{k-2+2}{2} = 3k
\text{ and }
\binom{3k-1+2}{2}-9\binom{k-1+2}{2} = -3k.
\]
In both cases, the values are positive in lower degrees. 
\end{proof}

We remark that the Hilbert series for fat point schemes in $\mathbb{P}^2$ supported also at $\leq 8$ points can be determined, with Guardo and Harbourne \cite{GH07} attributing these results, including also that for nine points, to Nagata \cite{NA60}.

\subsection{The case $n=4$ and $7 \leq m \leq 9$}

\begin{proposition} \label{prop:47}

For $d \geq 1$, write $28(d-1) = 13d_0 + d_1$, where $0\le d_1<13$. Then $d_0$ is the degree of the Hilbert series of $R_{4,7,d}$ and the leading coefficient is the coefficient of $t^{d_1}$ in 
\[\frac{(1-t^{1+13(d-1) - 6d_0})^7}{(1-t)^4},\]
which for $d\geq 73$ or $d_0\ge 15 d_1$, equals
\[\binom{d_1+3}{3}.\] 
\end{proposition}

\begin{proof}
Let $D$ be a non-negative integer. Cremona transformations gives 
\[
\begin{aligned}
&[4,d^7;D]  && k = 4(d-1) - 2D\\
&[4,(d+k)^3 d^4; D+k] && k^{(1)} = k-2D+3k-2k=2k \\
&[4,(d+2k)^3 (d+k)^3 d; D+3k] && k^{(2)}=k+7k-6k = 2k\\
&[4,(d+3k)^2(d+2k)^4(d+k);D+5k] && k^{(3)}= k+10k-10k=k\\
&[4,(d+3k)^4(d+2k)^3;D+6k] && k^{(4)}= k+12k-12k=k\\
& [4,(d+3k)^7;D+7k].\\
\end{aligned}
\]
We have  
\[ D+7k<0 \Longleftrightarrow 
 D + 7 \cdot (4(d-1)-2 D) < 0 \Longleftrightarrow
\left \lfloor \frac{28(d-1)}{13}  \right \rfloor
< D\Longleftrightarrow d_0 < D, \] 
which shows that the degree of the Hilbert series is $d_0$. To compute the leading coefficient, we let $D=d_0$ which gives $k = 4(d-1)-2d_0 = (13d_0+d_1)/7-2d_0 = (d_1-d_0)/7$, and $d_0+7k = d_1$. We need the dimension of 
\[
[4,(d+3k)^7;d_0+7k] = [4,(d+3k)^7;d_1]
\]
and since \[d+3k = 1+\frac{13d_0+d_1}{28}+\frac{3(d_1-d_0}{7} = 1+\frac{d_0+13d_1}{28}
\] 
we have that 
\[
d+3k>d_1 \Longleftrightarrow \frac{d_0+13d_1}{28} \ge d_1 \Longleftrightarrow d_0 \ge 15d_1.
\]
which proves that the leading coefficient is 
\[
\binom{d_1+3}{3}
\]
when $d_0\ge 15 d_1$ and this inequality holds for all $d\ge 73$. When $d \leq 72$, we have that 
$d+3k \leq 12,$ and it is computationally feasible to use a Macaulay2 computation with the ideal generated by the $d+3k$'th powers of  
$x_1,x_2,x_3,x_4,x_1+x_2+x_3+x_4,
2x_1+3x_2+5x_3+7x_4,
11x_1+17x_2+19x_3+29x_4
$
to check that until the corresponding degree $d_0 + 7k$,
the Hilbert series is non-zero and agrees with the series predicted by Fr\"oberg for generic forms of degree $d+3k=1+13(d-1)-6d_0.$
\end{proof}

We note that the series for $R_{4,7,d}$ agrees with the Fr\"oberg conjecture until $d \leq 7$.

\begin{proposition} \label{prop:48}
The Hilbert series of $R_{4,8,d}$ equals the series predicted by Fr\"oberg för $8$ general forms of degree $d$. In particular, the degree of the Hilbert series of 
$R_{4,8,d}$ is $2(d-1)$, and the leading coefficient equals $d$.
\end{proposition}

\begin{proof}
Oneto shows in his thesis \cite[Theorem 2.7.2]{ON16} that the Hilbert series for the generic fat point ideal 
\[
\mathfrak{p}_1^{(d)} \cap  \cdots \cap \mathfrak{p}_8^{(d)}
\]
is minimal, with gives that the Hilbert series for 
$R_{4,8,d}$ is minimal, that is, given by the Fr\"oberg conjecture. It follows that the coefficient of $t^{2d-2}$ is given by 
\[\binom{1+2d}{3} - 8 \binom{1+d}{3} = d.\]
It is straightforward to check that 
$\binom{1+2d-i}{3} - 8 \binom{1+d-i}{3}$ is positive for $i>0$, and that it is negative for $i=-1,$ so
$2d-2$ is indeed the degree of the series.

\end{proof}

\begin{proposition}
The degree of the Hilbert series of 
$R_{4,9,d}$ is $2(d-1)$, and the leading coefficient is bounded from below by $1$.
\end{proposition}

\begin{proof}
The upper bound is given by Proposition \ref{prop:48}. For the coefficient, consider $R_{4,9,2}$. The codimension in degree $3$ is equal to one, so there is a quadratic form $f$ in the inverse system. It follows that $f^{d-1}$ is in the inverse system related to $R_{4,9,d}$, which gives the claim.

\end{proof}
See also \cite{DL07} for results on seven points in $\mathbb{P}^3$, and \cite{BDP16} for results on eight and nine points in $\mathbb{P}^3$.

\subsection{The case $n=5,m=8$}

We begin by the following lemma.

\begin{lemma} \label{lemma:increasingd}
Suppose that the coefficient of the series of 
$R/(\ell_1^d,\ldots,\ell_m^d)$ in degree $i<2d$ is the one predicted by Fr\"oberg, and non-zero. Then so is the coefficient of the series of $R/(\ell_1^{d+1},\ldots,\ell_m^{d+1})$ in degree $i.$
\end{lemma}

\begin{proof}
Suppose that there is a non-trivial relation
\[F_1 \ell_1^{d+1} + \cdots + F_m \ell_m^{d+1}=0 \]
with $\deg(F_i)<d-1$. Then
\[(F_1 \ell_1)\ell_1^d + \cdots + (F_m \ell_m) \ell_m^d=0, \]
and since $\deg(F_i \ell_i) < d$ for $i=1,\ldots,m$, this relation is also non-trivial. 

Thus, if there are no non-trivial relations among $\ell_1^d,\ldots,\ell_m^d$ in degree $i$, then there are no non-trivial relations among $\ell_1^{d+1},\ldots,\ell_m^{d+1}$ in degree $i$, and the claim follows.
\end{proof}

\begin{proposition}
Let $d \geq 1$ and write $80(d-1)= 31 d_0 + d_1$ with 
$0 \leq d_1 \leq 30.$ Then $d_0$ is the degree of the Hilbert series of $R_{5,8,d}$ and when $d_0 \ge 49 d_1$, which holds for all $d \geq 541$, the leading coefficient equals 
\[c=\binom{d_1+4}{4},\] 
while if $d_0 < 49$, the leading coefficient equals the coefficient of 
$t^{d_1}$ in 
\[\frac{(1-t^{1+31(d-1) - 12d_0})^8}{(1-t)^5},\]
except for the cases given by the following table.

\begin{table}[ht]
\centering
\[
\begin{array}{|c|c||c|c|}
\hline
d & c & d & c \\
\hline
2  & 7           & 13 & [456,1016] \\
4  & 50          & 14 & [405,413] \\
5  & [21,29]     & 16 & [810,850] \\
6  & 168         & 18 & [1457,1557] \\
7  & [56,96]     & 23 & [1855,1863] \\
9  & [126,246]   & 25 & [3140,3180] \\
11 & [251,531]   &    & \\
\hline
\end{array}
\]
\end{table}

\end{proposition}

\begin{proof}
Let $D$ be a positive integer. Cremona transformations gives 
\[
\begin{aligned}
& [5,d^8;D] && k = 5(d-1)-2D \\
& [5,(d+k)^3d^5; D+k], && k^{(1)} = k + 3k - 2k = 2k \\
& [5,(d+2k)^3(d+k)^3d^2; D+3k], && k^{(2)} = k + 8k - 6k = 3k \\
& [5,(d+4k)^1(d+3k)^2 (d+2k)^2 (d+k)^2; D+6k], && k^{(3)} = k + 14k - 12k = 3k \\
& [5,(d+5k)^1(d+4k)^3 (d+3k)^2 (d+2k)^1; D+9k], && k^{(4)} = k + 20k - 18k = 3k \\
& [5,(d+6k)^1(d+5k)^3 (d+4k)^3 (d+3k)^1; D+12k], && k^{(5)} = k + 25k - 24k = 2k \\
& [5,(d+6k)^3(d+5k)^4 (d+4k)^1; D+14k], && k^{(6)} = k + 28k - 28k = k \\
& [5,(d+6k)^5(d+5k)^3; D+15k] && k^{(7)} = k + 30k - 30k = k \\
& [5,(d+6k)^8; D+16k]
\end{aligned}
\]
We have that 
\[ D+16k<0 \Longleftrightarrow 
 D + 16 \cdot (5(d-1)-2 D) < 0 \Longleftrightarrow
\left \lfloor \frac{80(d-1)}{31}  \right \rfloor
< D\Longleftrightarrow d_0 < D,\] 
which proves that $d_0$ is the degree of the Hilbert series of $R_{5,8,d}$.
For the leading coefficient, we need consider the case when $D=d_0$ and $k = 5(d-1)-2D = (d_1-d_0)/16$ and we need to compute the dimension of the linear system 
\[[5,(d+6k)^8;d_0+16k] = [5,(d+6k)^8;d_1].
\]
For this we note that 
\[
d+6k = 1 + \frac{31d_0+d_1}{80}+ \frac{6(d_1-d_0)}{16} = 1+\frac{d_0+31d_1}{80},\] 
so that  
\[
d+6k>d_1 \Longleftrightarrow 
\frac{d_0+31d_1}{80} \ge d_1
\Longleftrightarrow 
 d_0\ge  49 d_1,
\]
and we  conclude that the leading coefficient is given by 
\[
c = \binom{d_1+4}{4}, \quad \text{when $d_0\ge 49d_1$}.
\]
We can check that the inequality $d_0\ge 49d_1$ holds for all $d \geq 541$. 

Suppose now that $d_0 < 49 d_1$. The Cremona reduced degree $d_1$ is the remainder of $80(d-1)$ when divided by $31$, while the quotient, $d_0$, is increasing with $d$. It will now be convenient to work with the representation $d-1 = 31a + b$ for $0 \leq b \leq 30$, and we get 
 \[d+6k=31a+b+1 + 6 \cdot \left(5 \cdot (31a+b)-2\left \lfloor \frac{80(31a+b)}{31}  \right \rfloor \right)= a + 1 + 31b  - 12\left \lfloor \frac{80b}{31}  \right \rfloor.\]
For $b\notin\{12,24\}$ we have that this expression equals $d+6k = a+1+(7b)\%12$, where $m\%k$ denotes the remainder of $m$ when divided by $k$, while for $b\in \{12,24\}$, we have to add $12$. For all $b$, we have that $d_0+16k = d_1 = (18b)\%31$.

In the cases $b \in B_0= \{0,2,7,9,14,16,19,21,23,26,28,30\}$, we check for $a=0$ with Macaulay2 that the coefficients are the claimed ones. Moreover, it holds that $d_0+16k < 2 (d+6k)$ in these cases. Thus Lemma \ref{lemma:increasingd} applies, and the statement follows for any $a$.
        
We continue to check the cases $b \in \{0,\ldots,30\} \setminus B_0$ 
for $a=1$. In all these cases we check the coefficient with Macaulay2, and draw conclusions for any $a \geq 1$ for $b \in B_1=\{4,6,11,13,18,20,25,27\}$ using Lemma \ref{lemma:increasingd}. 

For $b \in B \setminus (B_0 \cup B_1)$ we consider $a=2$, check that the coefficient is the claimed one, and now  Lemma \ref{lemma:increasingd} applies in the cases $b \in B_2 = \{1,3,8,10,15,17,22,24,29\}$.

For $b \in B \setminus (B_0 \cup B_1 \cup B_2) = \{5,12 \}$ we consider $a=3$, check that the coefficient is the claimed on, and use Lemma \ref{lemma:increasingd}.

Thus the cases that remains to examine are $B \setminus B_0$ for $a=0$, that is, 
\[d\in \{ 2,4,5,6,7,9,11,12,13,14,16,18,19,21,23,25,26,28,30\}.\] For the cases $d \in 
\{12,19,21,26,28,30\}
$, we check that the coefficient is the expected one. For the cases $d \in \{2,4,6\}$ we can check that $R_{5,8,d}$ has the series predicted by Fr\"oberg, but that the corresponding Cremona transform, which increases the degree in these cases, does not give a coefficient which agrees with Fr\"oberg in this degree.
For the remaining cases, we do a Cremona transformation (which is the identity for the cases $d \in \{5,7,9,11,13\}$), and make a computation for a specialization 
which gives an upper bound, and verify Fr\"oberg's conjecture in each such degree, which gives the corresponding lower bound.
\end{proof}

We remark that it is possible to determine the coefficients in full by using the generators of the Cox-Nagata ring for $n=5$ and $8$ general linear forms, but it would require explicit descriptions of them.

\subsection{The case $n=6, m = 9$}

\begin{proposition} \label{prop:69}
The degree of the Hilbert series of $R_{6,9,d}$ is equal to $3(d-1)$ and the leading coefficient is bounded from below by $d$.
\end{proposition}

\begin{proof}
We first show that the inverse system 
\[ [6,d^9;3(d-1)+1], k=-2\]
is the zero set.
A Cremona transformation gives
\[ [6,(d-2)^3d^6;3(d-1)-1], k' = -4.\]
When $d \leq 2$, we get non-positive exponents, so in these cases, the inverse system is the zero set.  
Another Cremona transformation gives 
\[ [6,(d-4)^3 (d-2)^3 d^3;3(d-1)-5]. \] 
If $d =3,4$ we again get non-positive exponents, so we can assume that $d \geq 5$. It is enough to show that the inverse system in one degree below is the zero set, so we consider \[ [6,(d-4)^3 (d-2)^3 d^3;3(d-1)-6],  k^{(1)}= -6 \]. 
\subsubsection*{Claim:} After $i$ Cremona transformations starting from \[ [6,(d-4)^3 (d-2)^3 d^3;3(d-1)-6],\]  we get 
\[ [6,(d-2(i+2))^3 (d-2(i+1))^3 (d-2i)^3;3(d-1)-6 \cdot(i+1)]\] with $k^{(i)}= -6.$

We prove the claim by induction on $i$, with the base case already being settled. For the induction step we get 
\[D^{(i+1)} = 3(d-1) - 6 (i + 1) - k^{(i)} = 3(d-1) - 6 (i + 2),\]
and
\[k^{(i+1)} = 3(d-2(i+2)) + 3(d-2(i+1)) - 6 - 2\cdot (3(d-1) - 6 (i + 2) )= -6,\]
which proves the claim. It follows that 
the exponents eventually will become non-positive, which gives that the degree of the Hilbert series of $R_{6,9,d}$ is at most $3(d-1)$.

We now turn to the dimension in degree $3(d-1)$. The claim is trivial for $d=1$. For $d=2$ a computation gives that the series agrees with that predicted by Fr\"oberg for generic forms. In particular, there are two linearly independent cubic forms $f,g$ such that $
\ell_i^2 \circ f=\ell_i^2 \circ g=0$ for $i=1,\ldots,9.$ It follows from the structure of the Cox-Nagata ring that 
\[ \ell_i^d \circ f^{d-j} g^{j-1} = 0 \text{ for } i = 1,\ldots,9 \text{ and } j=1,\ldots,d.\] Since $f$ and $g$ are linearly independent, so is the set $\{f^{d-j} g^{j-1} \}$, which finishes the proof.
\end{proof}

We believe that the bound on the dimension given in Proposition \ref{prop:69} is sharp, but have only checked it computationally for $d \leq \compBoundsixnine$. 

We have now determined $D_{n,n+3,d}$ for $n=3,4,5,6$. For the case $n \leq 2$, when the Hilbert series of $R_{n,n+3,d}$ agrees with the one predicted by Fr\"oberg, it is straightforward to check that 
 $D_{1,m,d} = d-1$, and $D_{2,m,d} = \left \lfloor \frac{(m+2)(d-1)}{m+1} \right \rfloor.$ We summarize these observations in Table \ref{table:degree}, where we for reference also include the corresponding values for the well known cases $m=n$ and $m=n+1$.

\begin{table}[ht]
\centering
\begin{tabular}{|c||c|c|c|c|c|c|}
\hline
$m\backslash n$ & $1$ & $2$ & $3$ & $4$ & $5$ & $6$ \\
\hline \hline
$n$ 
& $d-1$
& $2(d-1)$
& $3(d-1)$
& $4(d-1)$
& $5(d-1)$
& $6(d-1)$ \\
\hline 
$n+1$
& $d-1$
& $\left\lfloor \frac{3(d-1)}{2} \right\rfloor$
& $2(d-1)$
& $\left\lfloor \frac{5(d-1)}{2} \right\rfloor$
& $3(d-1)$
& $\left\lfloor \frac{7(d-1)}{2} \right\rfloor$ \\
\hline
$n+2$
& $d-1$
& $\left\lfloor \frac{4(d-1)}{3} \right\rfloor$
& $2(d-1)$
& $\left\lfloor \frac{12(d-1)}{5} \right\rfloor$
& $3(d-1)$
& $\left\lfloor \frac{24(d-1)}{7} \right\rfloor$ \\
\hline
$n+3$
& $d-1$
& $\left\lfloor \frac{5(d-1)}{4} \right\rfloor$
& $\left\lfloor \frac{12(d-1)}{7} \right\rfloor$
& $\left\lfloor \frac{28(d-1)}{13} \right\rfloor$
& $\left\lfloor \frac{80(d-1)}{31} \right\rfloor$
& $3(d-1)$ \\
\hline
$n+4$
& $d-1$
& $\left\lfloor \frac{6(d-1)}{5} \right\rfloor$
& $\left\lfloor \frac{21(d-1)}{13} \right\rfloor$
& $2(d-1)$
& 
&  \\
\hline
$n+5$
& $d-1$
& $\left\lfloor \frac{7(d-1)}{6} \right\rfloor$
& $\left\lfloor \frac{48(d-1)}{31} \right\rfloor$
& $2(d-1)$
& 
&  \\
\hline
$n+6$
& $d-1$
& $\left\lfloor \frac{8(d-1)}{7} \right\rfloor$
& $\left\lfloor \frac{3(d-1)}{2} \right\rfloor$
& 
& 
&  \\
\hline
\end{tabular}
\caption{The degree of the Hilbert series for $R_{n,m,d}$ for some values of $n$ and $m$. In the cases $n \leq 2, m \leq n+1,$ and $(n,m)  \in \{(3,9),(4,8)\}$ the series corresponds to that given by Fr\"oberg for general forms. 
} \label{table:degree}
\end{table}
It is natural to ask how these sequences of numbers continues, and in the next section we will give one observation in this direction. 
\subsection{The case $n=k^2-3, m=k^2$}

\begin{proposition} 
Let $k \geq 2$. We have 
\[D_{k^2-3,k^2,d} \geq  \binom{k}{2}\cdot(d-1).\]
\end{proposition}

\begin{proof}

In the cases $k=2,3$, we have an equality by our earlier results. For the general case, consider the quadratic situation $d=2.$ Let $n = k^2-3.$

After a change of coordinates, we can assume that 
$\ell_i = x_i$ for $i = 1,\ldots,n, 
\ell_{n+1} = x_1+\cdots+x_n, \ell_{n+2} = a_1x_1+\cdots+a_nx_n, 
\ell_{n+3} = b_1x_1+\cdots+b_nx_n.$
Let $M$ be the matrix whose $i$'th row is
\[ (x_i m_1,\ldots, x_i m_s, m'_1, \ldots, m'_t), \]
where $m_1,\ldots, m_s$ are the monomials in $a_i$ and $b_i$ up to degree $k-2$, and where $m'_1,\ldots, m'_t$ are the monomials in $a_i$ and $b_i$ up to degree $k-1$, except for $1,a_i^{k-1},b_i^{k-1}$.
Thus $s=\binom{k}{2}$ and $t=\binom{k+1}{2}-3$, so 
$s+t=k^2-3$, so $M$ is a quadratic matrix. We have 
$x_i^2 \circ |M| = 0,$ since $M$ is squarefree, and 
\[\ell_{n+1}^2 \circ |M| = \ell_{n+2}^2 \circ |M| = \ell_{n+3}^2 \circ |M| =0\]
since when we replace $x_i$ by $1/a_i/b_i$ in two places, we get two equal columns. Thus \[D_{k^2-3,k^2,2} \geq  \binom{k}{2}.\]
It now follows by the pigeonhole principle that 
\[ \ell_i^d \circ |M|^{d-1} = 0 \text { for } i = 1,\ldots, n+3. \]

\end{proof}

\section{The conjecture by Harbourne, Schenck, and Seceleanu} 
We now turn to the following conjecture by Harbourne, Schenck, and Seceleanu \cite{HSS11} where some of the results from the previous section provide counterexamples. 

\begin{conjecture}[\cite{HSS11}] \label{conj:HSS} For $I = 
(\ell_1^t,\dots,\ell_n^t)\subseteq S=\kk[x_1,\dots,x_r]$ with $\ell_i\in S_1$ general and $n\ge r+1\ge 5$, WLP fails for all $t>>0$.
\end{conjecture}

In our language, this means that for $m > n \geq 4$, the algebra $R_{n,m,d}$ fails the WLP for all $d$ sufficiently large and our next proposition shows that this is not the case. 

\begin{proposition}
    \label{prop:wlp}
    The algebras $R_{4,k^2,d}$ and $R_{5,8,d}$ have the WLP for all $k\ge 2$ and all $d$.
\end{proposition} 

We will need the following lemma to prove this proposition. 
See \cite[Proposition 2.1 and Corollary 2.3]{MMN12} for similar results.

\begin{lemma} \label{lemma:wlp}
Let $\ell_1, \ldots, \ell_m$ be general linear forms in $\kk[x_1,\ldots,x_{n+1}],$ and let $\bar{\ell}_1, \ldots, \bar{\ell}_m$ be their images in $\kk[x_1,\ldots,x_n]=\kk[x_1,\ldots,x_{n+1}]/(x_{n+1}).$ Suppose that $m\geq n$ and that $A=\kk[x_1,\ldots,x_n]/(\bar{\ell}_1^{d_1}, \ldots, \bar{\ell}_m^{d_m})$ has the series predicted by Fr\"oberg. Then $B=\kk[x_1,\ldots,x_{n+1}]/(\bar{\ell}_1^{d_1}, \ldots, \bar{\ell}_m^{d_m})$ has the WLP, with $x_{n+1}$ as a WL element.
\end{lemma}

\begin{proof}
 We write \[
 \HS(A,t) = \sum_{i\ge 0} a_i t^i, \quad \HS(B,t)=\sum_{i\ge 0} b_it^i\quad\text{and}\quad \left[\frac{\prod (1-t^{d_i})}{(1-t)^{n+1}}\right] = \sum_i c_it^i.\]
We have $A \cong B/(x_{n+1})$ and hence $a_i\le b_i-b_{i-1}$ for all $i\ge 1$. Hence when $a_i=0$ we have that $b_i\ge b_{i-1}$ and the multiplication  $\times x_{n+1}:B_{i-1}\longrightarrow B_i$ is surjective. 

By assumption, we have that $\HS(A,t) = \left[\prod (1-t^{d_i})/(1-t)^{n}\right]$ and by  \cite[Lemma 4]{FR85} we have that 
\[
\HS(A,t) = \left[(1-t)\left[\frac{\prod_i (1-t^{d_i})}{(1-t)^{n+1}}\right]\right] = \sum_{i\ge 0} \max\{c_i-c_{i-1},0\} t^i.
\]
Hence when $a_i>0$, we have that $a_i = c_i-c_{i-1}$ and since $a_i\ge b_i-b_{i-1}$ we get $c_i-c_{i-1}\ge b_i-b_{i-1}$, which gives 
\[
b_i-c_i\le b_{i-1}-c_{i-1} \le \cdots \le b_0-c_0=0
\]
but the first non-zero difference $b_i-c_i$ has to be positive according to the main theorem in \cite{FR85}. We conclude that $b_i=c_i$ and that the multiplication $\times x_{n+1}\colon B_{i-1}\longrightarrow B_{i}$ is injective when $a_i>0$.
\end{proof}

We are now in a position to prove Proposition~\ref{prop:wlp} giving counterexamples to Conjecture~\ref{conj:HSS}.

\begin{proof}[Proof of Proposition~\ref{prop:wlp}]
By \cite{CM06} the algebras $R_{3,k^2,d}$ have the Hilbert series predicted by Fr\"oberg for all $k$ and all $d$. By Proposition \ref{prop:48} the algebras $R_{4,8,d}$ have the Hilbert series predicted by Fr\"oberg for all $d$. Hence, by Lemma~\ref{lemma:wlp}, the algebras $R_{4,k^2,d}$ and $R_{5,8,d}$ all have the WLP.
\end{proof}

\begin{remark}
Although we have proven that the Iarrobino conjecture does not hold for $n$ large enough, it is open for small values of $n$, in particular for $n=3$ and $m\geq 10$, where it coincides with the SHGH conjecture. However, if $R_{3,m,d}$ for $m \geq 10$ has the series for $m$ generic forms of degree $d$ in three variables for all $d$, then, by Lemma \ref{lemma:wlp}, the algebra $R_{4,m,d}$ has the WLP for all $d$, which would give more counterexamples to Conjecture~\ref{conj:HSS}. And the converse also holds -- if Conjecture~\ref{conj:HSS} holdsfor $n=4$ and some $m\geq 10$, this would mean that the SHGH conjecture is false.
\end{remark}

\subsection*{Acknowledgements}
   Computer experiments in Macaulay2 \cite{M2} was the starting point of this paper. Some of the figures were drawn with the help of an AI tool. We thank Alessandro Oneto for drawing our attention to \cite{LPS23}. The first author was partially funded by the Swedish Research Council VR2024-04853, and the second author was funded by the Swedish Research Council VR2022-04009.

\end{document}